\documentclass[12pt]{article}
\usepackage{amssymb}
\usepackage{amsmath}
\usepackage{amsthm}
\usepackage{color}
\usepackage{comment}
\usepackage{geometry}		
\usepackage{graphicx}
\usepackage[sort&compress]{natbib}     

\makeatletter
	\@addtoreset{equation}{section}
\makeatother
\renewcommand{\labelenumi}{\roman{enumi})}

\DeclareFontFamily{OT1}{pzc}{}
\DeclareFontShape{OT1}{pzc}{m}{it}{<-> s * [1.10] pzcmi7t}{}
\DeclareMathAlphabet{\mathpzc}{OT1}{pzc}{m}{it}

\let\originalleft\left
\let\originalright\right
\renewcommand{\left}{\mathopen{}\mathclose\bgroup\originalleft}
\renewcommand{\right}{\aftergroup\egroup\originalright}

\begin{document}

\newcommand{\bO}{{\bf 0}}
\newcommand\cL{\mathcal{L}}
\newcommand\cO{\mathcal{O}}
\newcommand\cP{\mathcal{P}}
\newcommand{\rD}{{\rm D}}
\newcommand{\re}{{\rm e}}
\newcommand{\ri}{{\rm i}}
\newcommand{\ee}{\varepsilon}

\newcommand{\myStep}[2]{{\bf Step #1} --- #2\\}

\newtheorem{theorem}{Theorem}[section]
\newtheorem{corollary}[theorem]{Corollary}
\newtheorem{lemma}[theorem]{Lemma}
\newtheorem{proposition}[theorem]{Proposition}

\theoremstyle{definition}
\newtheorem{definition}{Definition}[section]
\newtheorem{example}[definition]{Example}

\theoremstyle{remark}
\newtheorem{remark}{Remark}[section]



\title{
Boundary Hopf bifurcations in three-dimensional Filippov systems.
}
\author{
D.J.W.~Simpson\\\\
School of Mathematical and Computational Sciences\\
Massey University\\
Palmerston North, 4410\\
New Zealand
}

\maketitle


\begin{abstract}

For piecewise-smooth ordinary differential equations, the occurrence of a Hopf bifurcation on a switching surface is known as a boundary Hopf bifurcation. Boundary Hopf bifurcations are codimension-two, so occur at points in two-parameter bifurcation diagrams. From any such point there issues a curve of grazing bifurcations, where the limit cycle born in the Hopf bifurcation hits the switching surface. For Filippov systems, these are usually grazing-sliding bifurcations whose local dynamics are dictated by piecewise-linear maps. In general, these maps have many independent parameters and extraordinarily rich dynamical behaviour. We show that for three-dimensional Filippov systems only a two-parameter family of piecewise-linear maps is relevant, because sliding motion induces a loss of dimension, and the stability of the limit cycle is degenerate at the Hopf bifurcation. We derive explicit formulas for the two parameters in terms of quantities associated with the boundary Hopf bifurcation, and perform a comprehensive numerical analysis to characterise the attractor of the family, which may be chaotic. The results are illustrated with a pedagogical example, a pest control model, and a model of a food chain with threshold-based harvesting. To evaluate the parameters, we use a formula for the linear term of the discontinuity map associated with grazing-sliding bifurcations. In this paper we present a new, simpler derivation of this formula for $n$-dimensional systems based on displacements from a virtual counterpart.

\end{abstract}

\section{Introduction}
\label{sec:intro}

To determine how a differential equation model
gives different predictions under different conditions,
we identify critical parameter values, {\em bifurcations}, where the
dynamical behaviour of the model changes in a fundamental way.
Codimension-one bifurcations, such as saddle-node, period-doubling,
and Hopf bifurcations, are realised by varying a single parameter.
Codimension-two bifurcations, such as Bautin, Takens-Bogdanov,
and homoclinic-Hopf bifurcations, require the variation of two parameters.
Codimension-two bifurcations occur at points in two-parameter bifurcation diagrams,
and for each type of codimension-two bifurcation it is useful to understand its {\em unfolding},
i.e.~the way in which curves of codimension-one bifurcations issue from the codimension-two point.
This is because codimension-two bifurcations are often organising centres
whose unfoldings dictate the dynamics of the system across wide areas of parameter space. 

This paper studies an unfolding for ordinary differential equations of the form
\begin{equation}
\dot{X} = \begin{cases}
F_L(X;\nu,\eta), & H(X;\nu,\eta) < 0, \\
F_R(X;\nu,\eta), & H(X;\nu,\eta) > 0,
\end{cases}
\label{eq:F}
\end{equation}
where $X \in \mathbb{R}^3$ is the state variable,
and $\nu, \eta \in \mathbb{R}$ are parameters.
Such piecewise-smooth equations
are apt models of mechanical systems with impacts or friction,
engineered devices with on/off control,
and other phenomena that switch between distinct modes of evolution \cite{DiBu08}.
In general, the vector fields $F_L$ and $F_R$ differ on the switching surface
\begin{equation}
\Sigma_{\nu,\eta} = \left\{ X \in \mathbb{R}^3 \,\middle|\, H(X;\nu,\eta) = 0 \right\},
\label{eq:Sigma}
\end{equation}
which we refer to as the {\em discontinuity surface}.
Subsets of $\Sigma$ over which $F_L$ and $F_R$
are both directed toward $\Sigma$ are attracting sliding regions
on which we assume orbits evolve according to Filippov's convention \cite{Fi88,Je18b}.
For relay control systems, such motion represents
the limit of infinitely fast switching \cite{JoRa99}.
For mechanical systems with stick-slip friction modelled by Coulomb's law,
sliding motion represents sticking \cite{BlCz99,FeGu98}.
In ecology and economics, sliding motion may represent the situation
that species or companies hesitate over a decision
between different courses of action \cite{DeGr07,PuSu06}.

This paper concerns the codimension-two scenario that one piece of the system
has a Hopf bifurcation on $\Sigma$.
This is termed a boundary Hopf bifurcation, and its unfolding always involves a curve of Hopf bifurcations,
a curve of grazing bifurcations,
and a curve of boundary equilibrium bifurcations,
where the system has an equilibrium on $\Sigma$ \cite{DeDe11,DiPa08,GuSe11,Si10,SiKo09}.
The grazing bifurcations are where the Hopf cycle (limit cycle created in the Hopf bifurcation)
collides with $\Sigma$.
Other bifurcation curves can emanate from the codimension-two point,
relating to other invariant sets created in the boundary equilibrium and grazing bifurcations \cite{SiMe08}.

If $F_L = F_R$ on $\Sigma$, then the limit cycle persists
without change in stability through the grazing bifurcation \cite{DiBu01}.
If $F_L \ne F_R$ on $\Sigma$, then in generic situations the local dynamics interacts with a sliding region.
If this sliding region is repelling,
the grazing bifurcation destroys the limit cycle
because orbits are ejected into a different part of phase space \cite{JeCh10,JeHo11}.
If the sliding region is attracting,
the grazing bifurcation is a {\em grazing-sliding bifurcation}
whose local dynamics are described by
a continuous and asymptotically piecewise-linear Poincar\'e map, Fig.~\ref{fig:schem3}.

\begin{figure}[t!]
\begin{center}
\includegraphics[width=10cm]{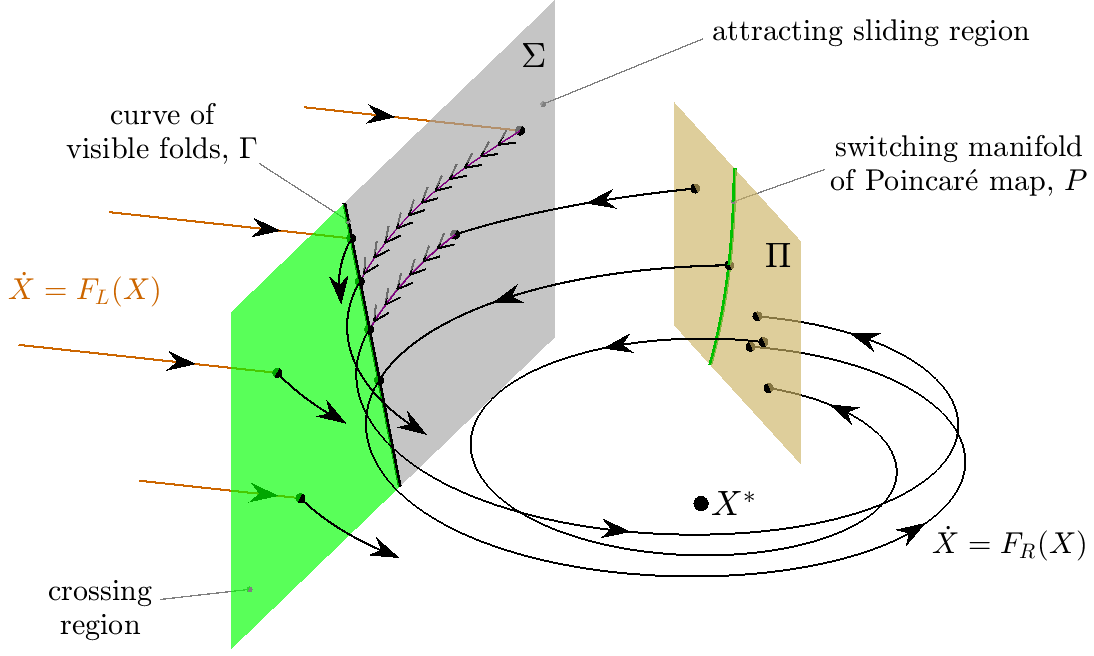}
\caption{
A phase portrait of a Filippov system \eqref{eq:F}
subject to assumptions listed in \S\ref{sec:thm}.
There is a single discontinuity surface $\Sigma$, defined by all points $X$ for which $H(X) = 0$.
To the left [right] of $\Sigma$ orbits evolve under $\dot{X} = F_L(X)$ [$\dot{X} = F_R(X)$].
The discontinuity surface is shaded green [turquoise] where it is a crossing region [attracting sliding region];
these regions are bounded by a curve $\Gamma$ of visible folds.
Sliding motion is indicated with multitudinous small arrows.
The surface $\Pi$ is used to define a Poincar\'e map $P$;
this map is piecewise-smooth with pieces $P_L$ (corresponding to orbits with a sliding segment)
and $P_R$ (corresponding to orbits that do not reach $\Sigma$).
\label{fig:schem3}
} 
\end{center}
\end{figure}

This paper analyses the grazing-sliding case for three-dimensional Filippov systems.
The basic unfolding of this scenario is shown in Fig.~\ref{fig:schemBifSet}.
The Poincar\'e map is two-dimensional,
and, by a result of Nusse and Yorke \cite{NuYo92},
in generic situations its leading-order terms can be converted via a change of variables to the form
\begin{equation}
\begin{bmatrix} x \\ y \end{bmatrix} \mapsto
\begin{cases}
\begin{bmatrix} \tau_L x + y + \mu \\ -\delta_L x \end{bmatrix}, & x \le 0, \\[4mm]
\begin{bmatrix} \tau_R x + y + \mu \\ -\delta_R x \end{bmatrix}, & x \ge 0.
\end{cases}
\label{eq:bcnf}
\end{equation}
The family \eqref{eq:bcnf} is the two-dimensional border-collision normal form.
It contains four parameters $\tau_L, \delta_L, \tau_R, \delta_R \in \mathbb{R}$,
in addition to the border-collision parameter $\mu$ whose value can be scaled to $-1$ or $1$
corresponding to the two sides of the bifurcation.
The border-collision normal form has remarkably rich dynamics
that yet to be fully catagorised \cite{AvSc12,BaGr99,FaSi23,GhMc24,Gl16e,SuGa06,ZhMo08}.

\begin{figure}[t!]
\begin{center}
\includegraphics[width=15.6cm]{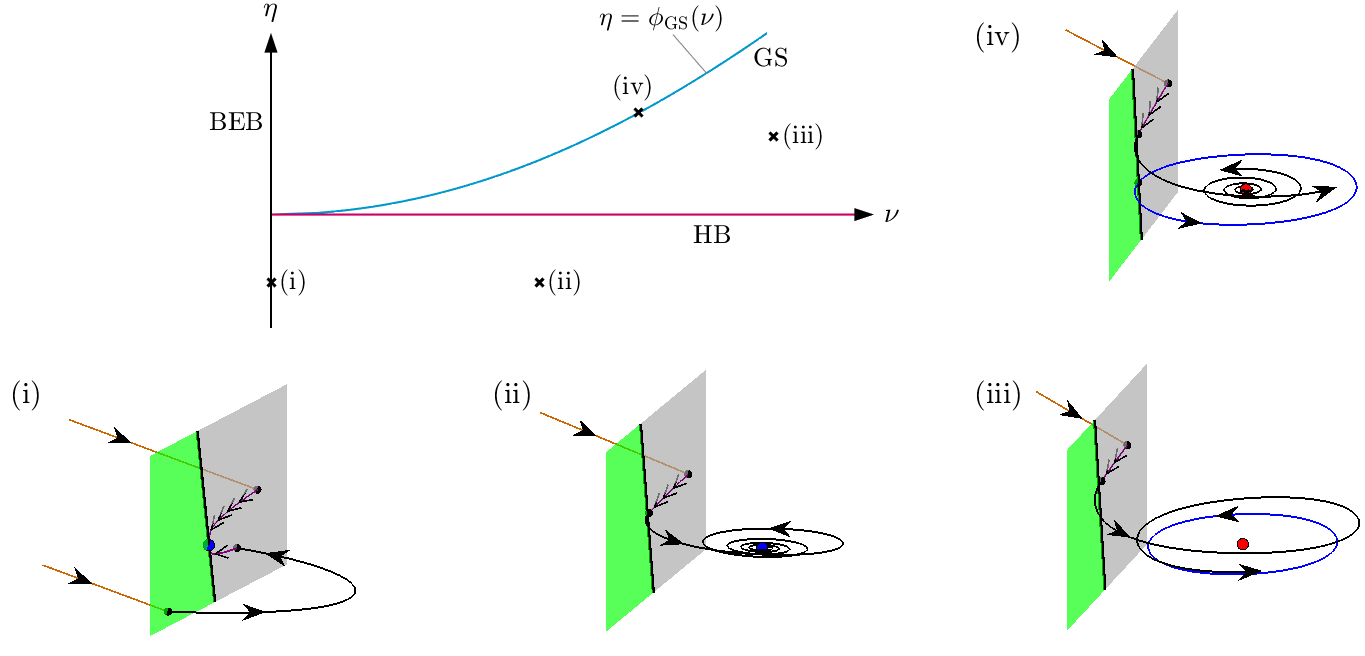}
\caption{
A sketch illustrating the basic unfolding of a boundary Hopf bifurcation of a
Filippov system of the form \eqref{eq:F}
(BEB: boundary equilibrium bifurcation;
HB: Hopf bifurcation;
GS: grazing-sliding bifurcation).
The first two bifurcation curves coincide with the $\eta$ and $\nu$ axes, as in Theorem \ref{th:main}.
The equilibrium $X^*(\nu,\eta)$ is a zero of the right vector field $F_R$,
the Hopf bifurcation is supercritical,
and the sliding region is attracting.
Four representative phase portraits are included:
(i) on the BEB curve $X^*(\nu,\eta)$ belongs to the discontinuity surface;
(ii) in the lower-right quadrant $X^*(\nu,\eta)$ is stable and admissible;
(iii) immediately above the HB curve $X^*(\nu,\eta)$ is unstable and the Hopf cycle does not intersect the discontinuity surface;
(iv) on the GS curve $\eta = \phi_{\rm GS}(\nu)$ the limit cycle grazes the discontinuity surface.
\label{fig:schemBifSet}
} 
\end{center}
\end{figure}

Each point on the grazing curve corresponds to the border-collision normal form
with some values of $\tau_L$, $\delta_L$, $\tau_R$, and $\delta_R$
that vary continuously along the curve.
For clarity, suppose the left piece of the normal form corresponds to orbits with sliding segments,
and the right piece of the normal form corresponds to orbits that do not reach $\Sigma$.
Then $\delta_L = 0$, because orbits with sliding regions are forced to pass through the one-dimensional set $\Gamma$, see Fig.~\ref{fig:schem3}.
Also, there is a restriction on the values of $\tau_R$ and $\delta_R$ as we limit to the codimension-two point.
This is because the Hopf cycle has a Floquet multiplier tending to $1$ at the Hopf bifurcation.
The Hopf cycle corresponds to a fixed point of the right piece of \eqref{eq:bcnf},
thus $\delta_R = \tau_R - 1$ at the codimension-two point.

The remainder of this paper is organised as follows.
In \S\ref{sec:thm} we present Theorem \ref{th:main}.
This theorem states the above claims precisely
and provides formulas for the limiting values of the parameters
of the border-collision normal form at the codimension-two point.
The theorem necessarily contains a large number of non-degeneracy conditions
for ensuring the boundary equilibrium, Hopf, and grazing-sliding bifurcations are generic.
From a bifurcation theory perspective,
it is necessary and instructive to have these conditions characterised and formulated quantitatively.

Theorem \ref{th:main} is proved in \S\ref{sec:proof}.
This is achieved by leveraging earlier work to obtain the existence, uniqueness, and smoothness
of the three codimension-one bifurcation curves.
For the limiting parameter values of the normal form,
$\delta_R$ is the non-degenerate Floquet multiplier of the Hopf cycle,
while $\tau_L$ is constructed from the Poincar\'e map
and the linear term of the discontinuity map associated with grazing-sliding bifurcations.
In Appendix \ref{app:formula} we provide a novel derivation of this term
for grazing-sliding bifurcations in $n$-dimensions.

In \S\ref{sec:bcnf} we perform a numerical bifurcation analysis
of the normal form \eqref{eq:bcnf} with $\delta_L = 0$, $\delta_R = \tau_R - 1$,
and $\mu = -1$, corresponding to the dynamics beyond the grazing-sliding bifurcation.
We find that if the map has an attractor,
the attractor is either a fixed point,
a period-two solution, or is a finite union of line segments on which orbits are chaotic.
For the last case we partition parameter space
into regions according to the geometric structure of the chaotic attractor.

In \S\ref{sec:appl} we show how the results apply to mathematical models.
We first consider a purely pedagogical model.
We then consider the three-species food chain model of Hastings and Powell \cite{HaPo91}
with two different types of threshold control.
In each case we use Theorem \ref{th:main} to identify the limiting instance of \eqref{eq:bcnf}
that applies to the given boundary Hopf bifurcation.
We then apply the results of \S\ref{sec:bcnf} to explain
the nature of the attractor created in the grazing bifurcation near the codimension-two point.
By further analysing the normal form along the grazing bifurcation,
we identify additional codimension-two points
that give rise to other bifurcation curves.
A summary and discussion in provided in \S\ref{sec:conc}.

\section{Statement of the main result}
\label{sec:thm}

In this section we describe the unfolding precisely.
We first explain the non-degeneracy conditions required
for the codimension-one bifurcations to occur in a generic fashion,
then state the unfolding as Theorem \ref{th:main}.

\subsection{Two-piece systems}

We consider systems of the form \eqref{eq:F}
for which the switching function $H$ and the vector fields $F_L$ and $F_R$ are $C^3$.
By the regular value theorem \cite{Hi76},
the discontinuity surface $\Sigma_{\nu,\eta}$
is a two-dimensional $C^3$ manifold in a neighbourhood of any point on the surface
for which the gradient vector $\nabla H$ is non-zero.

The dynamics associated with a boundary Hopf bifurcation are local.
This means the dynamics occurs in a neighbourhood of a point on the discontinuity surface.
Thus Theorem \ref{th:main} can be applied to systems with more than two smooth pieces
by simply ignoring the pieces of the system that are not associated with the bifurcation.

\subsection{Sliding motion}

As in di Bernardo {\em et al.}~\cite{DiBu08},
we use Lie derivatives to describe the flow in relation to the discontinuity surface.
The first Lie derivatives of $H$ with respect to $F_L$ and $F_R$ are the scalar products
\begin{equation}
\begin{split}
\cL_{F_L} H(X;\nu,\eta) &= \nabla H(X;\nu,\eta)^{\sf T} F_L(X;\nu,\eta), \\
\cL_{F_R} H(X;\nu,\eta) &= \nabla H(X;\nu,\eta)^{\sf T} F_R(X;\nu,\eta).
\end{split}
\end{equation}
These give the value of $\frac{d H}{d t}$
for orbits following $F_L$ and $F_R$ respectively.
Thus orbits of the left piece of \eqref{eq:F} reach $\Sigma_{\nu,\eta}$ transversally
at points with $\cL_{F_L} H(X;\nu,\eta) > 0$,
while orbits of the right piece of \eqref{eq:F} reach $\Sigma_{\nu,\eta}$ transversally
at points with $\cL_{F_R} H(X;\nu,\eta) < 0$.

Subsets of $\Sigma_{\nu,\eta}$ for which
$\cL_{F_L} H(X;\nu,\eta) \cL_{F_R} H(X;\nu,\eta) > 0$
are {\em crossing regions} where 
orbits of \eqref{eq:F} pass from one side of $\Sigma_{\nu,\eta}$ to the other.
Subsets of $\Sigma_{\nu,\eta}$ for which
$\cL_{F_L} H(X;\nu,\eta) > 0$ and $\cL_{F_R} H(X;\nu,\eta) < 0$ are {\em attracting sliding regions}
on which we assume orbits evolve under the sliding vector field
\begin{equation}
F_S(X;\nu,\eta) = \frac{F_R(X;\nu,\eta) \cL_{F_L} H(X;\nu,\eta) - F_L(X;\nu,\eta) \cL_{F_R} H(X;\nu,\eta)}
{\cL_{F_L} H(X;\nu,\eta) - \cL_{F_R} H(X;\nu,\eta)}.
\label{eq:FS}
\end{equation}
This vector field is the unique convex combination of $F_L$ and $F_R$ that is tangent to $\Sigma_{\nu,\eta}$,
and with this convention \eqref{eq:F} is a {\em Filippov system}.
Subsets of $\Sigma_{\nu,\eta}$ for which
$\cL_{F_L} H(X;\nu,\eta) < 0$ and $\cL_{F_R} H(X;\nu,\eta) > 0$ are {\em repelling sliding regions}.

\subsection{An equilibrium}

Suppose the right piece of \eqref{eq:F} has an equilibrium
$X^*(\nu,\eta)$ with associated eigenvalues $\alpha(\nu,\eta) \pm \ri \beta(\nu,\eta) \in \mathbb{C}$
and $\gamma(\nu,\eta) \in \mathbb{R}$.
The purpose of this paper is to unfold
the codimension-two scenario that $X^*(\nu,\eta) \in \Sigma_{\nu,\eta}$ and $\alpha(\nu,\eta) = 0$.

Without loss of generality, we assume this scenario occurs at $(\nu,\eta) = (0,0)$.
Let $X_0 = X^*(0,0)$, $\alpha_0 = \alpha(0,0)$, $\beta_0 = \beta(0,0)$, and $\gamma_0 = \gamma(0,0)$.
By assumption, $X_0 \in \Sigma_{0,0}$ and $\alpha_0 = 0$.
For genericity, we require $\beta_0 \ne 0$ and $\gamma_0 \ne 0$,
and without loss of generality can assume $\beta_0 > 0$.
That is, we assume the eigenvalues obey the genericity constraints
\begin{align}
\beta_0 &> 0, &
\gamma_0 &\ne 0.
\label{eq:aEigs}
\end{align}
So that Theorem \ref{th:main} can be stated plainly, we additionally assume
\begin{align}
X^*(0,\eta) &\in \Sigma_{0,\eta} \,, \qquad \text{for all $\eta$ in a neighbourhood of $0$}, \label{eq:aBEB} \\
\alpha(\nu,0) &= 0, \qquad \text{for all $\nu$ in a neighbourhood of $0$}. \label{eq:aHB}
\end{align}
With these conditions, and certain genericity constraints described below,
$\nu = 0$ is a line of boundary equilibrium bifurcations,
while $\eta = 0$ is a line of Hopf bifurcations, as in Fig.~\ref{fig:schemBifSet}.
For a general system with two parameters,
conditions \eqref{eq:aBEB} and \eqref{eq:aHB} can be imposed via a change of coordinates.
However, as shown in \S\ref{sec:appl},
Theorem \ref{th:main} can often be applied without performing a coordinate change.

Notice $H(X^*(0,0);0,0) = 0$ and $F_R(X^*(0,0);0,0) = \bO$ (the zero vector).
Let
\begin{equation}
\begin{split}
u &= \nabla H(X^*(0,0);0,0), \\
d &= F_L(X^*(0,0);0,0), \\
M_R &= \rD F_R(X^*(0,0);0,0).
\end{split}
\label{eq:udMR}
\end{equation}
The matrix $M_R$ is the Jacobian matrix associated with the equilibrium.
This matrix is invertible,
because by assumption it has eigenvalues $\pm \ri \beta_0 \ne 0$ and $\gamma_0 \ne 0$.
The vector $u$ is normal to $\Sigma_{0,0}$ at the equilibrium,
while the vector $d$ is the value of the left piece of \eqref{eq:F} at the equilibrium.
In Theorem \ref{th:main} we assume
\begin{equation}
u^{\sf T} d > 0,
\label{eq:avL}
\end{equation}
to ensure the existence of an attracting sliding region locally.
With instead $u^{\sf T} d < 0$, the sliding region is repelling
and grazing bifurcations lead to the absence of a local attractor.

\subsection{Boundary equilibrium bifurcations}

The point $X^*(\nu,\eta)$ is only an equilibrium of \eqref{eq:F} if $H(X^*(\nu,\eta);\nu,\eta) > 0$,
in which case we say it is {\em admissible}.
If instead $H(X^*(\nu,\eta);\nu,\eta) < 0$, we say $X^*(\nu,\eta)$ is {\em virtual}.

In order for the line $\nu = 0$ to correspond to generic boundary equilibrium bifurcations,
the equilibrium $X^*(0,0)$ must cross $\Sigma_{\nu,\eta}$ transversally
as the value of $\nu$ is varied through zero.
Without loss of generality, we assume $X^*(\nu,\eta)$ is admissible for $\nu > 0$, and virtual for $\nu < 0$.
In this case the transversality condition on the equilibrium is
\begin{equation}
\frac{\partial H(X^*(\nu,\eta);\nu,\eta)}{\partial \nu} \bigg|_{(0,0)} > 0.
\label{eq:anu}
\end{equation}

Boundary equilibrium bifurcations
also involve a {\em pseudo-equilibrium} (zero of the sliding vector field $F_S$) \cite{DiBu08}.
The pseudo-equilibrium is admissible if it belongs to a sliding region (attracting or repelling),
and virtual if it belongs to a crossing region.

In generic situations, the two equilibria are each admissible on exactly one side of the
boundary equilibrium bifurcation.
If they are admissible on different sides,
the bifurcation is referred to as {\em persistence},
while if they are admissible on the same side,
the bifurcation is referred to as a {\em nonsmooth-fold}.
By a result of di Bernardo {\em et al.}~\cite{DiNo08},
persistence occurs if $u^{\sf T} M_R^{-1} d < 0$,
and a nonsmooth-fold occurs if $u^{\sf T} M_R^{-1} d > 0$.

\subsection{Hopf bifurcations}

By \eqref{eq:aHB}, $X^*(\nu,\eta)$ has purely imaginary eigenvalues when $\eta = 0$.
Two genericity conditions are required for generic Hopf bifurcations to occur along the line $\eta = 0$,
in a neighbourhood of $\nu = 0$.
For transversality, we require
\begin{equation}
\frac{\partial \alpha}{\partial \eta}(0,0) > 0,
\label{eq:aeta}
\end{equation}
where the positive sign has been chosen without loss of generality.
For a unique limit cycle, we require
\begin{equation}
\chi_{\rm HB} \ne 0,
\label{eq:aLimitCycle}
\end{equation}
where $\chi_{\rm HB}$ is the non-degeneracy coefficient for the Hopf bifurcation.
The value of $\chi_{\rm HB}$ depends on the first, second, and third spatial derivatives
of $F_R$ evaluated at $(X;\nu,\eta) = (X^*(0,0);0,0)$ \cite{Gl99,Ku04}.
If $\chi_{\rm HB} < 0$, the Hopf bifurcation is {\em supercritical}
and creates a limit cycle for $\eta > 0$.
If $\chi_{\rm HB} > 0$, the Hopf bifurcation is {\em subcritical}
and creates a limit cycle for $\eta < 0$.

The limit cycle is stable only if the Hopf bifurcation is supercritical and $\gamma_0 < 0$.
As $(\nu,\eta) \to (0,0)$, the Floquet multipliers of the limit cycle
tend to $1$, $1$, and $\re^\frac{2 \pi \gamma_0}{\beta_0}$, see \eqref{eq:DQglobal}.

\subsection{Grazing-sliding bifurcations}

Let $w^{\sf T}$ and $v$ be left and right eigenvectors
of $M_R$ corresponding to the eigenvalue $\gamma_0$,
and normalised by $w^{\sf T} v = 1$.
The vector $w$ is orthogonal to any center manifold that
contains the limit cycle.
Consequently, we assume
\begin{equation}
\text{$u$ and $w$ are linearly independent},
\label{eq:auw}
\end{equation}
so that this manifold intersects the discontinuity surface transversally.

Now consider a two-dimensional Poincar\'e map $P$ defined from a Poincar\'e section
$\Pi \subset \mathbb{R}^3$, as illustrated in Fig.~\ref{fig:schem3}.
The particular choice of $\Pi$ is not important,
but it needs to be $C^3$, in the right half space $H > 0$, and transverse to the flow.
We can assume $P$ is piecewise-smooth,
with a piece $P_L$ corresponding to orbits with sliding segments,
and a piece $P_R$ corresponding to orbits that do not reach $\Sigma_{\nu,\eta}$ before returning to $\Pi$.
The switching manifold of $P$ corresponds to orbits that graze $\Sigma_{\nu,\eta}$.
Here $P_L = P_R$, hence $P$ is continuous.
The piece $P_R$ is $C^3$ because $F_R$ and $\Pi$ are $C^3$ \cite{Me07},
while $P_L$ is $C^1$ but usually not $C^2$ \cite{DiKo02}.

By truncating $P_L$ and $P_R$ to first order,
we obtain a piecewise-linear continuous map that approximates $P$.
Typically the truncated map reproduces the dynamics of $P$
quantitatively near the grazing-sliding bifurcation.
The truncated map can be converted to the normal form \eqref{eq:bcnf}
via an affine change of coordinates if an observability condition holds \cite{Si16}.

Importantly, the parameter values $\tau_L$, $\delta_L$, $\tau_R$, and $\delta_R$ of the normal form
can be evaluated without executing such a change of coordinates.
This is because these values are the traces and determinants of $\rD P_L$ and $\rD P_R$
evaluated at the grazing-sliding bifurcation.
Moreover, they are continuous functions of $\nu$, and
\begin{equation}
\delta_L(\nu) = 0, \qquad \text{for all $\nu > 0$},
\label{eq:}
\end{equation}
because the range of $P_L$ is one-dimensional.
Theorem \ref{th:main} provides formulas for the limiting values of the remaining parameters.

\subsection{Theorem statement}

We now state the main result.
Notice the boundary Hopf bifurcation unfolded by Theorem \ref{th:main} is codimension-two
because \eqref{eq:aBEB} and \eqref{eq:aHB} are independent codimension-one constraints.

\begin{theorem}
Consider a system \eqref{eq:F} where $F_L$, $F_R$, and $H$ are $C^3$,
and the right piece of the system has an equilibrium $X^*(\nu,\eta)$
satisfying \eqref{eq:aBEB} and \eqref{eq:aHB}.
Suppose the genericity conditions \eqref{eq:aEigs} and \eqref{eq:avL}--\eqref{eq:auw} are met.
Then there exists $\delta > 0$ and a unique $C^2$ function $\phi_{\rm GS} : [0,\delta) \to \mathbb{R}$,
with $\phi_{\rm GS}(0) = 0$, $\frac{d \phi_{\rm GS}}{d \nu}(0) = 0$,
and ${\rm sgn} \left( \frac{d^2 \phi_{\rm GS}}{d \nu^2}(0) \right) = -{\rm sgn} \left( \chi_{\rm HB} \right)$, such that
\begin{enumerate}
\renewcommand{\labelenumi}{\arabic{enumi})}
\item
persistence [nonsmooth-fold] BEBs occur on $\nu = 0$ with $-\delta < \eta < \delta$
if $u^{\sf T} M_R^{-1} d < 0$ [$u^{\sf T} M_R^{-1} d > 0$];
\item
supercritical [subcritical] Hopf bifurcations occur on $\eta = 0$ with $0 < \nu < \delta$
if $\chi_{\rm HB} < 0$ [$\chi_{\rm HB} > 0$];
\item
grazing-sliding bifurcations occur on $\eta = \phi_{\rm GS}(\nu)$ with $0 < \nu < \delta$.
\end{enumerate}
Moreover, with $P_L$ and $P_R$ defined as above,
\begin{align}
\lim_{\nu \to 0^+} \tau_L(\nu) &= \frac{u^{\sf T} v w^{\sf T} d}{u^{\sf T} d} \left( 1 - \re^{\frac{2 \pi \gamma_0}{\beta_0}} \right)
+ \re^{\frac{2 \pi \gamma_0}{\beta_0}}, \label{eq:tL0} \\
\lim_{\nu \to 0^+} \tau_R(\nu) &= \re^{\frac{2 \pi \gamma_0}{\beta_0}} + 1, \label{eq:tR0} \\
\lim_{\nu \to 0^+} \delta_R(\nu) &= \re^{\frac{2 \pi \gamma_0}{\beta_0}}, \label{eq:dR0}
\end{align}
where $\tau_L(\nu)$ and $\delta_L(\nu) = 0$ are the trace and determinant
of $\rD P_L$ evaluated at grazing-sliding bifurcation,
and $\tau_R(\nu)$ and $\delta_R(\nu)$ are the trace and determinant
of $\rD P_R$ evaluated at grazing-sliding bifurcation.
\label{th:main}
\end{theorem}

In summary, there are six genericity conditions:
\begin{itemize}
\setlength{\itemsep}{0pt}
\item 
\eqref{eq:aEigs} ensures the equilibrium does not have a zero eigenvalue,
\item 
\eqref{eq:avL} ensures the existence of an attracting sliding region,
\item 
\eqref{eq:anu} is the transversality condition for the boundary equilibrium bifurcation,
\item 
\eqref{eq:aeta} is the transversality condition for the Hopf bifurcation,
\item 
\eqref{eq:aLimitCycle} is the non-degeneracy condition for the Hopf bifurcation, and
\item 
\eqref{eq:auw} is the transversality condition for the grazing-sliding bifurcation.
\end{itemize}

\section{Proof of the main result}
\label{sec:proof}

In this section we prove Theorem \ref{th:main} in four steps.
In Step 1 we appeal to prior publications to verify items 1, 2, and 3 of Theorem \ref{th:main}.
Then in Step 2 we perform a coordinate change that blows up
an $\cO(\nu)$ neighbourhood of the limit cycle into an $\cO(1)$ region.
This allows us to study the Poincar\'e map in the limit $\nu \to 0$.
We choose the coordinate change in a way
that converts the Jacobian matrix $M_R$ to real Jordan form,
as this greatly simplifies our later calculations of the flow.

In Step 3 we introduce a second Poincar\'e section
and express the Poincar\'e map as a composition.
Specifically we use a {\em discontinuity map} \cite{DiBu08}
so that sliding segments can be treated as a local correction
about the point at which the limit cycle grazes the discontinuity surface.

In Step 4 we compute the derivatives of the two pieces of the Poincar\'e map,
and evaluate their traces and determinants.
For the piece of the map corresponding to orbits that do not meet $\Sigma$,
the derivative is computed from an explicit formula for the flow in the limit $\nu \to 0$.
For the piece of the map corresponding to orbits with sliding segments,
the derivative is computed using asymptotic calculations associated with the local correction.

\begin{proof}[Proof of Theorem \ref{th:main}]~\\
\myStep{1}{The three codimension-one bifurcation curves.}
Item 1 of Theorem \ref{th:main} is a consequence of direct calculations
of the pseudo-equilibrium, for details see \cite{DiNo08,Si18d}.
Item 2 of Theorem \ref{th:main}
is effectively a restatement of the Hopf bifurcation theorem \cite{Gl99,Ku04}.
The Hopf bifurcations occur for $\nu > 0$, because by \eqref{eq:anu}
the equilibrium $X^*(\nu,\eta)$ is admissible for small $\nu > 0$.

By the Hopf bifurcation theorem, the diameter of the limit cycle
is asymptotically proportional to $\sqrt{|\eta|}$.
By \eqref{eq:anu}, the distance of $X^*(\nu,\eta)$ from $\Sigma_{\nu,\eta}$ is asymptotically to $\nu$.
Also, by \eqref{eq:auw}, at first order the limit cycle does not grow parallel to $\Sigma_{\nu,\eta}$
as $\eta$ is varied from zero.
It follows that the limit cycle grazes $\Sigma_{\nu,\eta}$ on a curve $\eta = \phi_{\rm GS}(\nu)$
that is quadratically tangent to $\eta = 0$ at $(\nu,\eta) = (0,0)$.
Formally this demonstrated in \cite{Si10,SiKo09} where $\phi_{\rm GS}$
is shown to be $C^2$ via the implicit function theorem.

The Hopf cycle exists for small $\eta > 0$ if $\chi_{\rm HB} < 0$,
and small $\eta < 0$ if $\chi_{\rm HB} > 0$.
This is because by \eqref{eq:aeta}
the complex eigenvalues associated with the equilibrium 
are stable for $\eta < 0$, and unstable for $\eta > 0$.
Therefore $\frac{d^2 \phi_{\rm GS}}{d \nu^2}(0) > 0$ if $\chi_{\rm HB} < 0$,
and $\frac{d^2 \phi_{\rm GS}}{d \nu^2}(0) < 0$ if $\chi_{\rm HB} > 0$.

For the remainder of the proof we consider $\eta = \phi_{\rm GS}(\nu)$, where $\nu > 0$ is small.

\myStep{2}{A spatial blow-up and conversion to real Jordan form.}
Recall, $M_R$ has eigenvalues $\pm \ri \beta_0$ and $\gamma_0$.
Thus there exists invertible $T \in \mathbb{R}^{3 \times 3}$ such that
\begin{equation}
\tilde{M}_R = T^{-1} M_R T = \begin{bmatrix}
0 & \beta_0 & 0 \\
-\beta_0 & 0 & 0 \\
0 & 0 & \gamma_0
\end{bmatrix},
\label{eq:realJordanForm}
\end{equation}
is in real Jordan form.
Let
\begin{equation}
e_3 = \begin{bmatrix} 0 \\ 0 \\ 1 \end{bmatrix},
\nonumber
\end{equation}
and recall that $w^{\sf T}$ and $v$ are left and right eigenvectors of $M_R$ for the eigenvalue $\gamma_0$.
By \eqref{eq:realJordanForm}, $w^{\sf T}$ is a scalar multiple of $e_3^{\sf T} T^{-1}$,
while $v$ is a scalar multiple of $T e_3$.
By assumption $w^{\sf T} v = 1$, thus
\begin{equation}
v w^{\sf T} = T e_3 e_3^{\sf T} T^{-1}.
\label{eq:dyadicProduct}
\end{equation}
Also $u^{\sf T} T$ and $e_3^{\sf T}$ are linearly independent by \eqref{eq:auw}.

We now apply the affine coordinate change
\begin{equation}
Y = \frac{1}{\nu} \,T^{-1} \left( X - X^*(\nu,\phi_{\rm GS}(\nu)) \right),
\label{eq:coordChange}
\end{equation}
to the system \eqref{eq:F}.
This shifts the equilibrium to the origin,
blows up phase space by a factor $\frac{1}{\nu}$,
and puts the Jacobian matrix of the origin in real Jordan form in the limit $\nu \to 0$,
see Fig.~\ref{fig:schem1}.
We write the system in $Y$-coordinates as
\begin{equation}
\dot{Y} = \begin{cases}
\tilde{F}_L(Y;\nu), & \tilde{H}(Y;\nu) < 0, \\
\tilde{F}_R(Y;\nu), & \tilde{H}(Y;\nu) > 0,
\end{cases}
\label{eq:Ftransformed}
\end{equation}
where $\tilde{H} = \frac{H}{\nu}$.
By \eqref{eq:udMR},
\begin{align}
\tilde{F}_L(Y;\nu) &= \frac{1}{\nu} \left( T^{-1} d + \cO(\nu) \right), \label{eq:tildeFL2} \\
\tilde{F}_R(Y;\nu) &= \tilde{M}_R Y + \cO(\nu), \label{eq:tildeFR2} \\
\tilde{H}(Y;\nu) &= u^{\sf T} T Y + \psi + \cO(\nu), \label{eq:tildeH2}
\end{align}
where
\begin{equation}
\psi = \frac{\partial H(X^*(\nu,\eta);\nu,\eta)}{\partial \nu} \bigg|_{(0,0)}
\label{eq:anu2}
\end{equation}
is positive by \eqref{eq:anu}.

Let $\tilde{\Sigma}_\nu$, $\tilde{\Gamma}_\nu$, and $\tilde{\Pi}_\nu$ denote the $Y$-coordinate versions
of $\Sigma_{\nu,\phi_{\rm GS}(\nu)}$, $\Gamma_{\nu,\phi_{\rm GS}(\nu)}$,
and $\Pi_{\nu,\phi_{\rm GS}(\nu)}$, respectively, Fig.~\ref{fig:schem1}.
Let $Q$ be the transformed Poincar\'e map
defined by the first return of orbits of \eqref{eq:Ftransformed} to $\tilde{\Pi}_\nu$.
This map is two-dimensional, using coordinates on $\tilde{\Pi}_\nu$ that will be introduced in Step 4.
The map $Q$ is piecewise-smooth, with pieces $Q_L$ and $Q_R$
corresponding to orbits with sliding segments, and without sliding segments, respectively.
These pieces are the transformed versions of $P_L$ and $P_R$.
Thus, evaluated at the point at which the limit cycle grazes the discontinuity surface,
$\rD P_L$ and $\rD Q_L$ are similar,
and $\rD P_R$ and $\rD Q_R$ are similar.
Therefore the traces and determinants of $\rD P_L$ and $\rD P_R$
are identical to those of $\rD Q_L$ and $\rD Q_R$.

\begin{figure}[b!]
\begin{center}
\includegraphics[width=8cm]{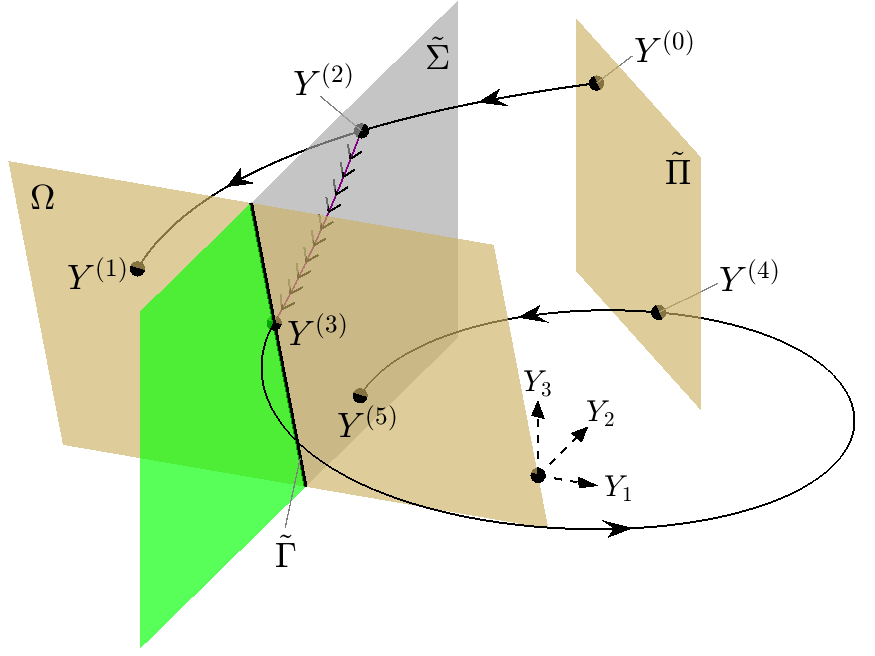}
\caption{
The phase space of \eqref{eq:F} in the $Y$-coordinate form \eqref{eq:Ftransformed}.
We show a sample orbit with a sliding segment,
and its virtual extension to the point $Y^{(1)} \in \Omega$.
\label{fig:schem1}
} 
\end{center}
\end{figure}

\myStep{3}{The Poincar\'e map $Q$ as a composition.}
We now introduce the additional cross-section of phase space
\begin{equation}
\Omega_\nu = \left\{ Y \in \mathbb{R}^3 \,\middle|\, \cL_{\tilde{F}_R} \tilde{H}(Y;\nu) = 0 \right\}.
\nonumber
\end{equation}
As illustrated in Fig.~\ref{fig:schem1},
$\Omega_\nu$ intersects the discontinuity surface $\tilde{\Sigma}_\nu$
along the fold curve $\tilde{\Gamma}_\nu$.
In a moment we will use $\Omega_\nu$ as the domain of a discontinuity map
and to express the Poincar\'e map as a composition.
First we compute $\Omega_\nu$ and its relation to the flow to first order.

By \eqref{eq:tildeFR2} and \eqref{eq:tildeH2},
\begin{equation}
\cL_{\tilde{F}_R} \tilde{H}(Y;\nu) = u^{\sf T} T \tilde{M}_R Y + \cO(\nu).
\nonumber
\end{equation}
Write
\begin{equation}
u^{\sf T} T = \begin{bmatrix} a & b & c \end{bmatrix},
\label{eq:abc}
\end{equation}
where $a,b,c \in \mathbb{R}$,
and notice $a^2 + b^2 \ne 0$ because
$u^{\sf T} T$ and $e_3^{\sf T}$ are linearly independent.
Since $\tilde{M}_R$ is given by \eqref{eq:realJordanForm}, we have
\begin{equation}
\cL_{\tilde{F}_R} \tilde{H}(Y;\nu) = -b \beta_0 Y_1 + a \beta_0 Y_2 + c \gamma_0 Y_3 + \cO(\nu).
\label{eq:tildevR2}
\end{equation}

Recall, we are considering $\eta = \phi_{\rm GS}(\nu)$,
so the Hopf cycle intersects the discontinuity surface at a single point.
Let $G(\nu)$ denote this point in $Y$-coordinates.
This point is a zero of $\tilde{H}$ and $\cL_{\tilde{F}_R} \tilde{H}$,
and belongs to the centre manifold $Y_3 = \cO(\nu)$ associated with the equilibrium $Y = \bO$.
So by using \eqref{eq:tildeH2} and \eqref{eq:tildevR2}, we obtain
\begin{equation}
G(\nu) = \frac{-\psi}{a^2 + b^2} \begin{bmatrix} a \\ b \\ 0 \end{bmatrix} + \cO(\nu).
\label{eq:grazPoint2}
\end{equation}

For sufficiently small $\nu > 0$, the gradient vector of \eqref{eq:tildevR2} at $Y = G(\nu)$ is non-zero,
thus $\Omega_\nu$ is a smooth two-dimensional manifold.
Specifically it is $C^2$ because $\tilde{F}_R$ and $\tilde{H}$ are $C^3$,
so $\cL_{\tilde{F}_R} \tilde{H}$ is $C^2$.
In a neighbourhood of $G(\nu)$,
the flow of $\dot{Y} = \tilde{F}_R(Y;\nu)$ intersects $\Omega_\nu$ transversally.
This is because by \eqref{eq:tildeFR2}, \eqref{eq:tildevR2}, and \eqref{eq:grazPoint2},
\begin{equation}
\cL^2_{\tilde{F}_R} \tilde{H}(G(\nu);\nu) = \beta_0^2 \psi + \cO(\nu),
\label{eq:aR}
\end{equation}
is non-zero for sufficiently small values of $\nu$.

Near the limit cycle,
the forward orbits of points on $\tilde{\Pi}_\nu$ pass through $\Omega_\nu$
before returning to $\tilde{\Pi}_\nu$.
For orbits following $\dot{Y} = \tilde{F}_R(Y;\nu)$,
let $Q_{\rm in}$ denote the map from $\tilde{\Pi}_\nu$ to $\Omega_\nu$,
and let $Q_{\rm out}$ denote the map from $\Omega_\nu$ to $\tilde{\Pi}_\nu$.
These maps are $C^2$, and
\begin{align}
Q_L &= Q_{\rm out} \circ Q_{{\rm disc},L} \circ Q_{\rm in} \,, &
Q_R &= Q_{\rm out} \circ Q_{\rm in} \,, &
\label{eq:QLQR}
\end{align}
where $Q_{{\rm disc},L}$ is a {\em discontinuity map} that accounts for sliding segments.
For the orbit shown in Fig.~\ref{fig:schem1},
$Q_{\rm in}$ maps $Y^{(0)}$ to $Y^{(1)}$,
$Q_{{\rm disc},L}$ maps $Y^{(1)}$ to $Y^{(3)}$,
and $Q_{\rm out}$ maps $Y^{(3)}$ to $Y^{(4)}$.
The point $Y^{(1)}$ does not belong to the orbit of \eqref{eq:Ftransformed};
it is virtual, constructed by smoothly extending $\tilde{F}_R$ beyond $\tilde{\Sigma}_\nu$.

Now let
\begin{equation}
Q_{\rm global} = Q_{\rm in} \circ Q_{\rm out} \,,
\nonumber
\end{equation}
which in Fig.~\ref{fig:schem1} takes $Y^{(3)}$ to $Y^{(5)}$.
By \eqref{eq:QLQR}, $Q_L$ and $Q_R$ are conjugate to $Q_{\rm global} \circ Q_{{\rm disc},L}$,
and $Q_{\rm global}$, respectively.
Thus $\rD Q_L$ and $\rD Q_R$ are similar to 
$\rD Q_{\rm global}(Q_{{\rm disc},L}) \rD Q_{{\rm disc},L}$ and $\rD Q_{\rm global}$, respectively.
Therefore
\begin{equation}
\begin{split}
\tau_L(\nu) &= {\rm trace} \left( \rD Q_{\rm global}(Q_{{\rm disc},L})
\rD Q_{{\rm disc},L} \middle) \right|_{Y = G(\nu)} \,, \\
\tau_R(\nu) &= {\rm trace} \left( \rD Q_{\rm global} \middle) \right|_{Y = G(\nu)} \,, \\
\delta_R(\nu) &= \det \left( \rD Q_{\rm global} \middle) \right|_{Y = G(\nu)} \,.
\end{split}
\label{eq:tauLtauRdeltaR}
\end{equation}
To the complete the proof it remains for us
to evaluate the right-hand sides of \eqref{eq:tauLtauRdeltaR} in the limit $\nu \to 0$.

\myStep{4}{Evaluation of the traces and determinants.}
For the remainder of the proof we suppose $a \ne 0$
so that we can use $(Y_1,Y_3)$ coordinates on $\Omega_\nu$
(if $a = 0$ the proof can be completed in the same
fashion using $(Y_2,Y_3)$ coordinates).
So we write $Q_{\rm global} = Q_{\rm global}(Y_1,Y_3;\nu)$,
and $Q_{{\rm disc},L} = Q_{\rm disc}(Y_1,Y_3;\nu)$.
Notice $\left( G(\nu)_1, G(\nu)_3 \right)$ is a fixed point of both maps.

We now evaluate the derivatives $\rD Q_{\rm global}$ and $\rD Q_{{\rm disc},L}$ at
$\left( G(\nu)_1, G(\nu)_3 \right)$ in the limit $\nu \to 0$.
To this end, we consider values of $Y_1$ and $Y_3$ that are slightly perturbed from $G(0)_1$ and $G(0)_3$,
and set $Y_2$ so that $Y \in \Omega_0$:
\begin{equation}
Y = \begin{bmatrix}
\frac{-a \psi}{a^2 + b^2} + \ee_1 \\
\frac{-b \psi}{a^2 + b^2} + \frac{b}{a} \,\ee_1 - \frac{c \gamma_0}{a \beta_0} \,\ee_3 \\
\ee_3
\end{bmatrix},
\label{eq:perturbedPoint}
\end{equation}
using the formulas \eqref{eq:tildevR2} and \eqref{eq:grazPoint2}.

Write $(Y_1',Y_3') = Q_{\rm global}(Y_1,Y_3;0)$.
The vector field $\tilde{F}_R$ is linear in the limit $\nu \to 0$,
so by using a closed-form expression for its flow, detailed in Appendix \ref{app:global}, we obtain
\begin{align}
\rD Q_{\rm global} \left( G(0)_1, G(0)_3; 0 \right)
&= \begin{bmatrix} \frac{\partial Y_1'}{\ee_1} & \frac{\partial Y_1'}{\ee_3} \\[1.2mm]
\frac{\partial Y_3'}{\ee_1} & \frac{\partial Y_3'}{\ee_3} \end{bmatrix} \nonumber \\
&= \begin{bmatrix} 1 & -\frac{b c \gamma_0}{\left( a^2 + b^2 \right) \beta_0} \left( 1 - \re^{\frac{2 \pi \gamma_0}{\beta_0}} \right) \\[1.2mm]
0 & \re^{\frac{2 \pi \gamma_0}{\beta_0}} \end{bmatrix}.
\label{eq:DQglobal}
\end{align}
By evaluating the trace and determinant of this matrix
we obtain \eqref{eq:tR0} and \eqref{eq:dR0}.

Now write $(Y_1',Y_3') = Q_{{\rm disc},L}(Y_1,Y_3;0)$.
An explicit formula for the linear term
of this type of discontinuity map is given as equation (8.74) in \cite{DiBu08},
but has a typo, and the formula, in our notation, should be
\begin{equation}
Y' = Y + Z(Y;0) \tilde{H}(Y;0) + \cO \left( \left( |\ee_1| + |\ee_3| \right)^{\frac{3}{2}} \right),
\label{eq:8p74}
\end{equation}
where $Z(Y;\nu)$ is the vector field
\begin{equation}
Z = \frac{\cL_{\tilde{F}_L} \cL_{\tilde{F}_R} \tilde{H}}
{\big( \cL_{\tilde{F}_R}^2 \tilde{H} \big) \big( \cL_{\tilde{F}_L} \tilde{H} \big)} \,\tilde{F}_R
- \frac{1}{\cL_{\tilde{F}_L} \tilde{H}} \,\tilde{F}_L \,,
\label{eq:8p74vector}
\end{equation}
as derived in Appendix \ref{app:formula}.
By evaluating the components of \eqref{eq:8p74vector}, see Appendix \ref{app:disc}, we obtain
\begin{align}
\rD Q_{\rm disc} \left( G(0)_1, G(0)_3; 0 \right)
&= \begin{bmatrix} \frac{\partial Y_1'}{\ee_1} & \frac{\partial Y_1'}{\ee_3} \\[1.2mm]
\frac{\partial Y_3'}{\ee_1} & \frac{\partial Y_3'}{\ee_3} \end{bmatrix} \nonumber \\
&= \begin{bmatrix}
\frac{(a \beta_0 - b \gamma_0) c r}{a \beta_0 (a p + b q + c r)} &
\frac{-(a \beta_0 - b \gamma_0) c}{\beta_0 \left( a^2 + b^2 \right)(a p + b q + c r)}
\left( a p + b q + \frac{b c r \gamma_0}{a \beta_0} \right) \\[1.2mm]
\frac{-\left( a^2 + b^2 \right) r}{a (a p + b q + c r)} &
1 - \frac{(a \beta_0 - b \gamma_0) c r}{a \beta_0 (a p + b q + c r)}
\end{bmatrix},
\label{eq:DQdisc}
\end{align}
writing also
\begin{equation}
T^{-1} d = \begin{bmatrix} p \\ q \\ r \end{bmatrix},
\label{eq:pqr}
\end{equation}
for the vector in \eqref{eq:tildeFL2}.
The trace of the product of \eqref{eq:DQglobal} and \eqref{eq:DQdisc} is
\begin{equation}
\lim_{\nu \to 0^+} \tau_L(\nu) 
= \frac{c r}{a p + b q + c r} \left( 1 - \re^{\frac{2 \pi \gamma_0}{\beta_0}} \right)
+ \re^{\frac{2 \pi \gamma_0}{\beta_0}}.
\label{eq:tL02}
\end{equation}
By \eqref{eq:abc} and \eqref{eq:pqr},
\begin{align}
a p + b q + c r &= u^{\sf T} T T^{-1} d = u^{\sf T} d, \nonumber \\
c r &= u^{\sf T} T e_3 e_3^{\sf T} T^{-1} c = u^{\sf T} v w^{\sf T} c, \nonumber
\end{align}
using also \eqref{eq:dyadicProduct}.
Thus \eqref{eq:tL02} is equivalent to \eqref{eq:tL0} as required.
\end{proof}

\section{A subfamily of the two-dimensional border-collision normal form}
\label{sec:bcnf}

Along the curve $\eta = \phi_{\rm GS}(\nu)$ of grazing-sliding bifurcations,
the dynamics created in the bifurcation are in generic situations described by the 
two-dimensional border-collision normal form \eqref{eq:bcnf}
with $\delta_L = 0$ and some values of $\tau_L$, $\tau_R$, and $\delta_R$.
The grazing limit cycle corresponds to a fixed point of the right piece of the map.
If this limit cycle is stable before grazing,
then $\delta_R > \tau_R - 1$, and the fixed point is admissible for $\mu > 0$.
In this case the dynamics after grazing correspond to $\mu < 0$,
and by scaling it suffices to set $\mu = -1$.

Taking $\nu \to 0$ in Theorem \ref{th:main} yields $\delta_R = \tau_R - 1$.
With also $\delta_L = 0$ and $\mu = -1$,
the normal form \eqref{eq:bcnf} reduces to the two-parameter family
\begin{equation}
\begin{bmatrix} x \\ y \end{bmatrix} \mapsto
\begin{cases}
\begin{bmatrix} \tau_L x + y - 1 \\ 0 \end{bmatrix}, & x \le 0, \\[4mm]
\begin{bmatrix} \tau_R x + y - 1 \\ (1-\tau_R) x \end{bmatrix}, & x \ge 0.
\end{cases}
\label{eq:reducedbcnf}
\end{equation}
Numerical explorations suggest that for all $(\tau_L,\tau_R) \in \mathbb{R}^2$,
\eqref{eq:reducedbcnf} has either a unique attractor or no attractor (i.e.~typical forward orbits diverge).
In this section we explain how the $(\tau_L,\tau_R)$-plane
divides into different regions according to the nature of the attractor.
In the case of chaotic attractors, the results are numerical.
Proofs are beyond the scope of this paper,
but could be achieved by following the techniques of \cite{GlSi21,Ko05,Mi80}
whereby polygonal trapping regions are constructed,
and the attractor is characterised as a certain piecewise-linear set
on which the dynamics are transitive.
Throughout this section we write
\begin{align}
g_L(x,y) &= \begin{bmatrix} \tau_L x + y - 1 \\ 0 \end{bmatrix}, &
g_R(x,y) &= \begin{bmatrix} \tau_R x + y - 1 \\ (1-\tau_R) x \end{bmatrix},
\nonumber
\end{align}
for the left and right pieces of \eqref{eq:reducedbcnf}.

\subsection{Overview}
\label{sub:overview}

\begin{figure}[b!]
\begin{center}
\setlength{\unitlength}{1cm}
\begin{picture}(15.6,13.1)
\put(0,5.6){\includegraphics[height=7.5cm]{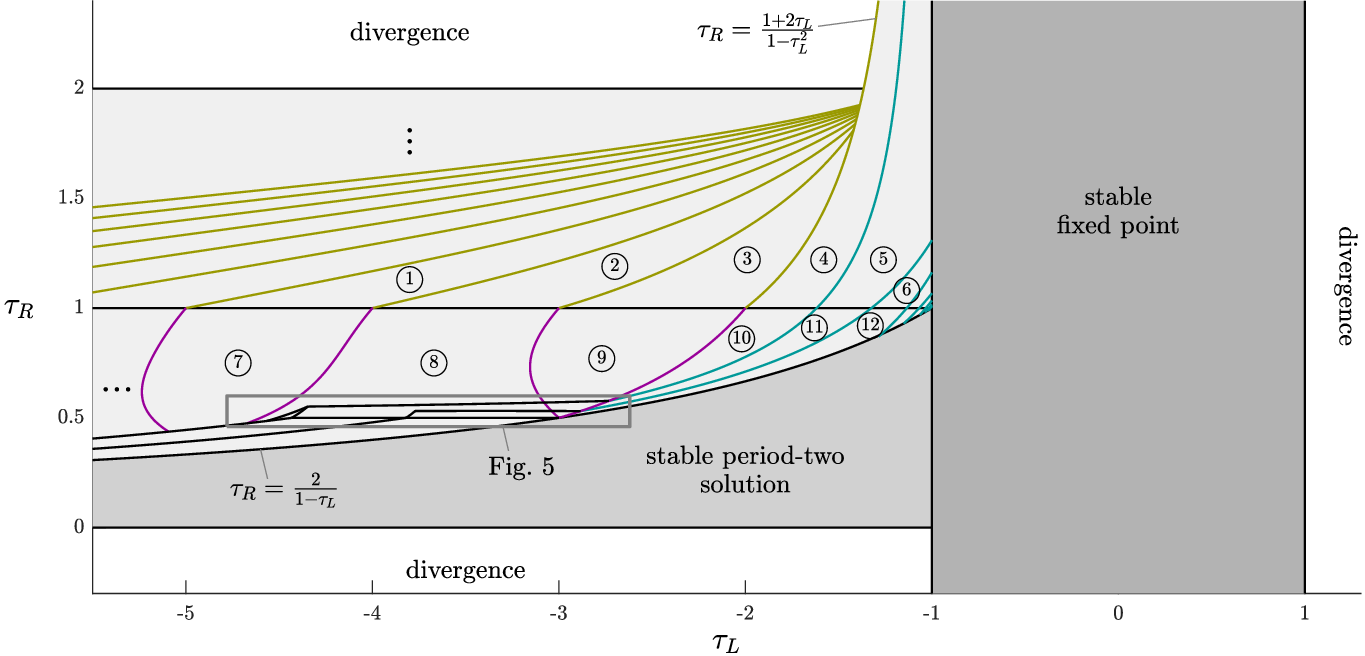}}
\put(0,0){\includegraphics[height=5.2cm]{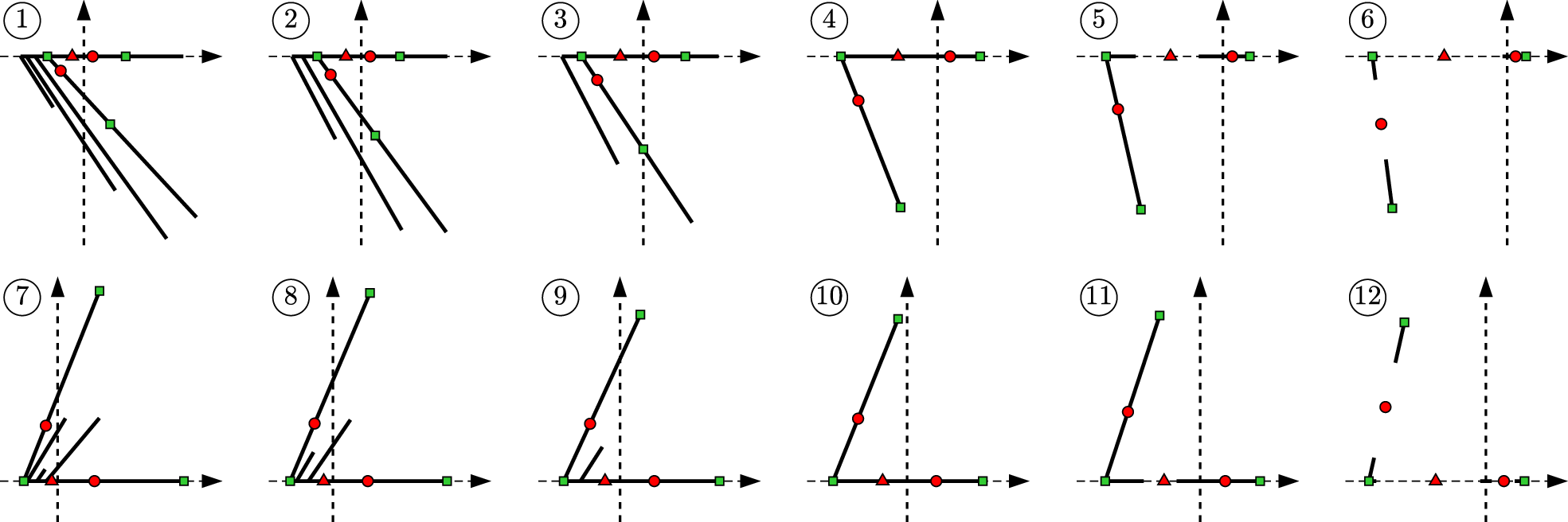}}
\end{picture}
\caption{
A two-parameter bifurcation diagram of the family \eqref{eq:reducedbcnf} and representative phase portraits.
For the twelve indicated regions we provide one phase portrait in $(x,y)$-coordinates
that shows the attractor (black),
the unstable fixed point \eqref{eq:fp} (red triangle),
an unstable period-two solution (red circles),
and the three points $U^{(1)}$, $U^{(2)}$, and $U^{(3)}$, given by \eqref{eq:threeImages} (green squares).
\label{fig:mapBifSetA}
} 
\end{center}
\end{figure}

\begin{figure}[b!]
\begin{center}
\setlength{\unitlength}{1cm}
\begin{picture}(13.6,10.6)
\put(0,5.6){\includegraphics[height=5cm]{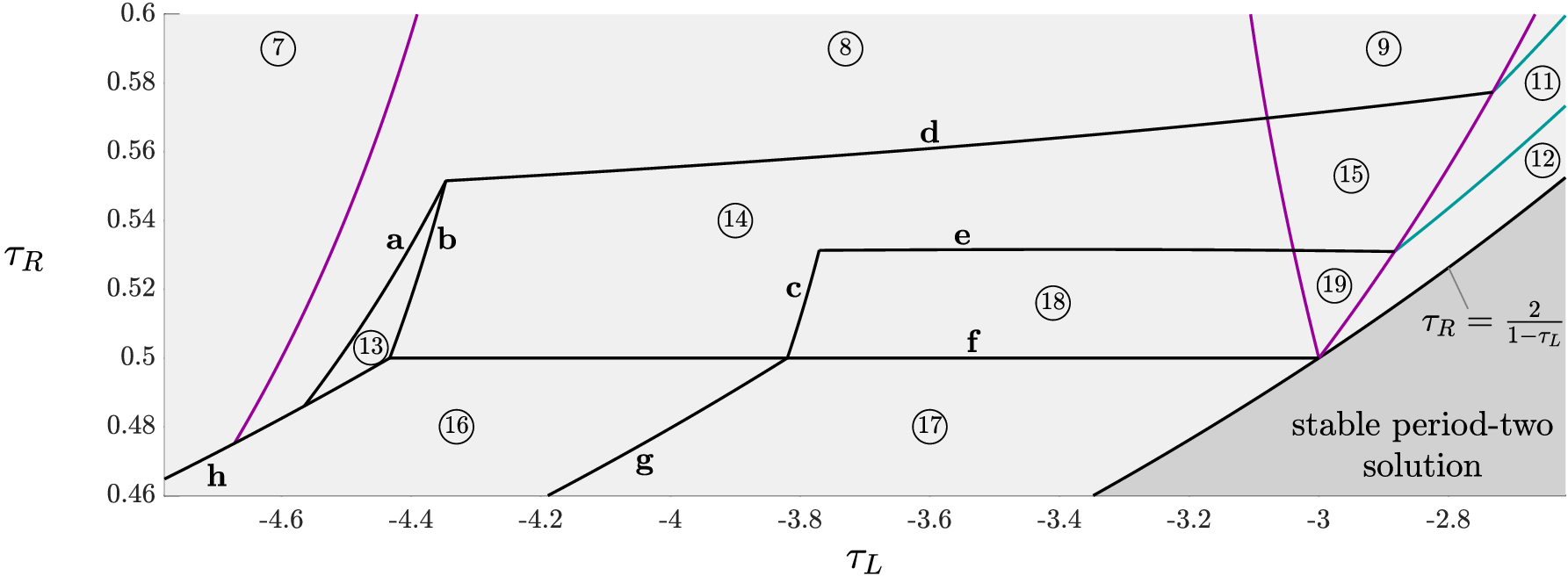}}
\put(0,0){\includegraphics[height=5.2cm]{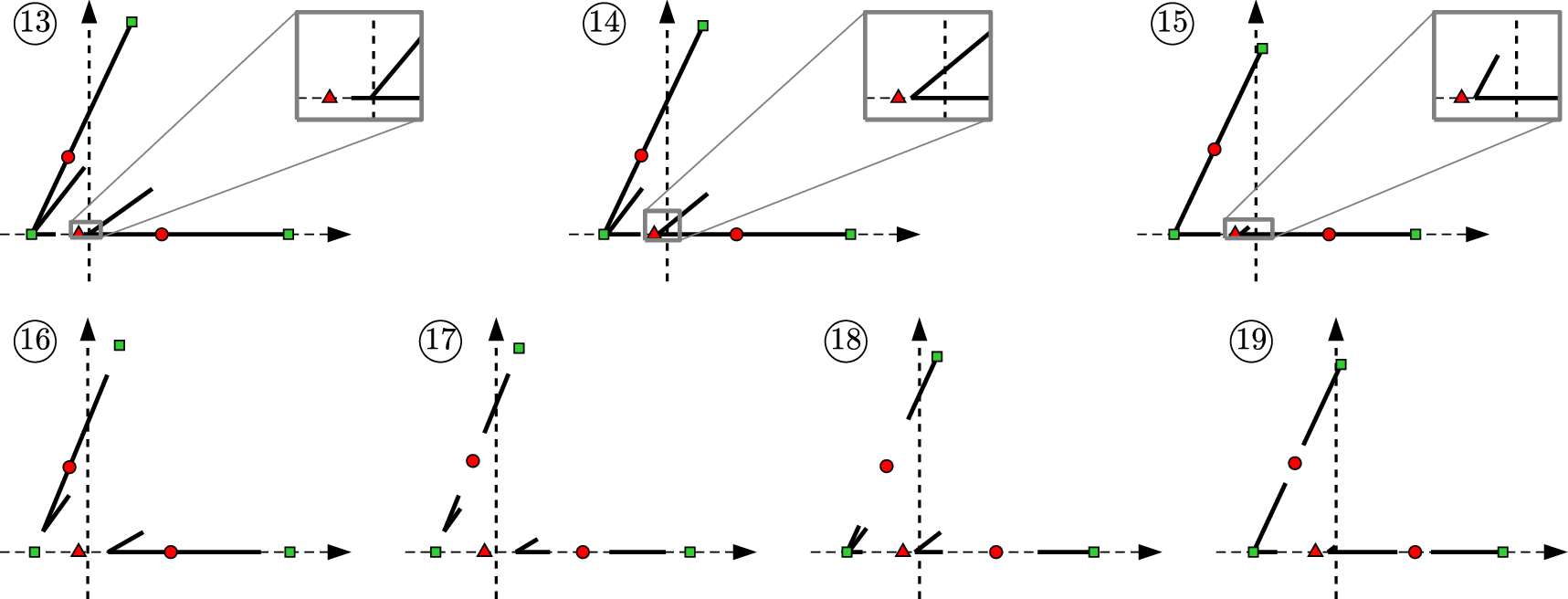}}
\end{picture}
\caption{
A magnification of Fig.~\ref{fig:mapBifSetA} and additional sample phase portraits.
The bifurcation curves {\bf a}--{\bf h} are described in \S\ref{sub:abc}.
\label{fig:mapBifSetB}
} 
\end{center}
\end{figure}

The division of the $(\tau_L,\tau_R)$-plane is shown in Fig.~\ref{fig:mapBifSetA},
and with a magnification in Fig.~\ref{fig:mapBifSetB}.
There are basic four scenarios (ignoring points on bifurcation curves).
First, if $-1 < \tau_L < 1$, then \eqref{eq:reducedbcnf} has the stable fixed point
\begin{equation}
P^* = \left( \tfrac{1}{\tau_L - 1}, 0 \right),
\label{eq:fp}
\end{equation}
belonging to the left half-plane.
Second, if $\tau_L < -1$ and $0 < \tau_R < \frac{2}{1 - \tau_L}$,
then \eqref{eq:reducedbcnf} has a stable period-two solution.
This solution is comprised of the points
\begin{align}
W_L &= \left( \tfrac{2}{\tau_L - 1}, -\tfrac{(\tau_L + 1)(\tau_R - 1)}{\tau_R (\tau_L - 1)} \right), &
W_R &= \left( \tfrac{\tau_L + 1}{\tau_R (\tau_L - 1)}, 0 \right),
\label{eq:WLWR}
\end{align}
where $W_L$ lies in the left half-plane,
and $W_R$ lies in the right half-plane.
Third, in the light grey area above the curve $\tau_R = \frac{2}{1 - \tau_L}$,
the attractor of \eqref{eq:reducedbcnf} is chaotic.
Elsewhere, \eqref{eq:reducedbcnf} has no attractor
and typical forward orbits diverge.

In the remainder of this section we consider the light grey area.
We have divided this area into regions according to the geometry of the attractor.
For $19$ of these regions, we provide in Figs.~\ref{fig:mapBifSetA} and \ref{fig:mapBifSetB}
a sample plot of the attractor.
In these plots, $P^*$ and $\left\{ W_L, W_R \right\}$ are unstable,
and indicated with a red triangle and red circles respectively.

If $\tau_R > 1$, then $g_R$ is orientation-preserving
and maps points with $x \ge 0$ to points with $y \le 0$,
while if $\tau_R < 1$, then $g_R$ is orientation-reversing
and maps points with $x \ge 0$ to points with $y \ge 0$.
Since $\delta_L = 0$, $g_L$ maps all points to the $x$-axis.
Consequently, the attractor contains a subset of the $x$-axis,
and is otherwise situated entirely below the $x$-axis if $\tau_R > 1$,
and above the $x$-axis if $\tau_R < 1$.

The attractor is always a finite union of line segments
(proved in some cases by Kowalcyzk \cite{Ko05}).
The simplest case is that the attractor is a union of two line segments, as in regions 4 and 10.
The endpoints of the line segments are the first three images of the origin:
\begin{equation}
\begin{split}
U^{(1)} &= g_L(0,0) = (-1,0), \\
U^{(2)} &= g_L^2(0,0) = (-\tau_L-1,0), \\
U^{(3)} &= \left( g_R \circ g_L^2 \right)(0,0) = \big( -\tau_R (\tau_L + 1) - 1, (\tau_R - 1)(\tau_L + 1) \big).
\end{split}
\label{eq:threeImages}
\end{equation}
These points are indicated by green squares in each of the 19 phase portraits,
and provide a useful sense of scale.

\subsection{Three bifurcation sequences}
\label{sub:threeBifSequences}

Fig.~\ref{fig:mapBifSetA} contains three sequences of bifurcation curves
coloured turquoise, olive, and purple.
As we cross a turquoise curve from left to right, the attractor splits into twice as many connected components.
So in regions 5 and 11, the attractor has two connected components,
while in regions 6 and 12, it has four connected components.
There are infinitely many of these curves converging to $(\tau_L,\tau_R) = (-1,1)$.
The first curve is where $g_L \left( U^{(3)} \right) = P^*$.
Here orbits in the attractor lose access to a neighbourhood of $P^*$,
so the attractor develops a hole centred at $P^*$.
As we pass through the curve, the size of the hole increases from zero in a continuous fashion,
so the attractor remains continuous with respect to Hausdorff metric \cite{GlSi20b}.
The remaining turquoise curves were constructed by renormalisation \cite{GhMc24,GhSi22}.

As we cross an olive curve from right to left, the attractor gains a limb (line segment).
There are infinitely many of these curves converging monotonically to $\tau_R = 2$.
The first ten of these curves are included in Fig.~\ref{fig:mapBifSetA}.
The first curve is where $g_L \left( U^{(3)} \right) = U^{(1)}$.
Here the attractor suddenly increases in size by gaining access to points left of $U^{(1)}$,
so is discontinuous with respect to Hausdorff metric.
For this reason, in region 3 the points $U^{(1)}$, $U^{(2)}$, and $U^{(3)}$ are now interior points
of the line segments that comprise attractor.
As we move further to the left (and/or upwards), each olive curve
is where one endpoint of the left-most limb of the attractor maps to the other endpoint of this limb,
and as we pass through the curve the attractor accumulates a new limb
by gaining access to more points on the $x$-axis and increasing in size discontinuously.

The purple curves form a similar sequence below $\tau_R = 1$.
The right-most purple curve is where $U^{(3)}$ lies on the switching line.
As we cross this curve the attractor gains a limb in a continuous fashion with respect to Hausdorff metric.
As we move left into region 8, then region 7, and so on, the attractor grows extra limbs.
The upper endpoints of these limbs are the points
$\left( g_R^k \circ g_L^2 \right)(0,0)$, with $k \ge 1$.
The purple curves are where these points lie on the switching line.

If the line segment with endpoints $U^{(1)}$ and $U^{(3)}$ intersects the switching line,
the intersection occurs at the point $\left( 0, \frac{1}{\tau_R} - 1 \right)$.
The image of this point under the map is
\begin{equation}
V = \left( \tfrac{1}{\tau_R} - 2, 0 \right),
\label{eq:V}
\end{equation}
which is a point that will significant below.
For the attractor in regions 7, 8, and 9,
$V$ is the left endpoint of the right-most limb of the attractor.

\subsection{Additional bifurcation curves}
\label{sub:abc}

Now refer to the magnification of Fig.~\ref{fig:mapBifSetA} shown in Fig.~\ref{fig:mapBifSetB}.
In regions 13--19 the attractor has more exotic geometries.
For example, in region 17 the attractor has four connected components.
Two components are line segments,
and two components are the union of two line segments.
The endpoints of the (six) line segments are
the zero-th through ninth images of $V$ under the map.

The bifurcation curves {\bf a}--{\bf h} in Fig.~\ref{fig:mapBifSetB}
are characterised by the following equations:
\begin{enumerate}
\renewcommand{\labelenumi}{{\bf \alph{enumi}}:}
\setlength{\itemsep}{0pt}
\item
$\left( g_L \circ g_R^2 \right) \left( U^{(3)} \right) = P^*$,
\item
$\left( g_L \circ g_R^2 \right) \left( U^{(3)} \right) = V$,
\item
$\left( g_L \circ g_R^2 \right) \left( U^{(3)} \right) = W_R$,
\item
$V = P^*$,
\item
$g_L^2(V) = W_R$,
\item
$V = (0,0)$,
\item
$\left( g_L \circ g_R^3 \circ g_L \circ g_R \right)(V) = W_R$,
\item
$\left( g_L \circ g_R^3 \circ g_L \circ g_R \right)(V) = V$.
\end{enumerate}
We do not have space to explain
how these constraints cause the attractor to change geometry.
In short, {\bf d}, {\bf a}, and part of {\bf h} effectively extend the boundary between regions 10 and 11
where the attractor splits into two components.
Similarly, {\bf e}, {\bf c}, and {\bf g} effectively extend the boundary between regions 11 and 12
where the attractor splits into four components.
As we go downwards through {\bf f},
the point $V$, which is a vertex of the attractor, passes through the origin.
Consequently, below {\bf f} the origin no longer belongs to the attractor.
For this reason, in regions 16 and 17 the attractor has become separated from
the first three iterates $U^{(1)}$, $U^{(2)}$, and $U^{(3)}$ of the origin.

\section{Unfolding boundary Hopf bifurcations in mathematical models}
\label{sec:appl}

In this section we study three examples.
For each example we numerically continue curves of one-parameter bifurcations,
and study the nature of the attractors near a boundary Hopf bifurcation.
We then evaluate the limiting parameter values of Theorem \ref{th:main}
and compare the theory of \S\ref{sec:bcnf} to the numerical observations.
We also consider the full curve of grazing-sliding bifurcations,
and determine numerically how the corresponding parameter values of border-collision normal form
vary along this curve.
By studying the dynamics of the normal form,
we identify codimension-two points on the grazing-sliding bifurcation curve
where the dynamics created in the bifurcation changes,
and other codimension-one bifurcations arise.

\subsection{A minimal example}

Consider the system
\begin{equation}
\dot{X} = \begin{cases}
\begin{bmatrix} 0 \\ -2 \\ 1 \end{bmatrix}, & \nu - X_1 - 2 X_2 - 3 X_3 < 0, \\
\begin{bmatrix} X_2 \\ -X_1 + \eta X_2 - X_2^3 \\ -\frac{1}{5} X_3 \end{bmatrix}, & \nu - X_1 - 2 X_2 - 3 X_3 > 0,
\end{cases}
\label{eq:toyF}
\end{equation}
where $\eta, \nu \in \mathbb{R}$ are parameters.
This system has the general form \eqref{eq:F} with
\begin{equation}
\begin{split}
F_L(X;\nu,\eta) &= \begin{bmatrix} 0 \\ -2 \\ 1 \end{bmatrix}, \\
F_R(X;\nu,\eta) &= \begin{bmatrix} X_2 \\ -X_1 + \eta X_2 - X_2^3 \\ -\frac{1}{5} X_3 \end{bmatrix}, \\
H(X;\nu,\eta) &= \nu - X_1 - 2 X_2 - 3 X_3 \,.
\end{split}
\nonumber
\end{equation}
We now show that \eqref{eq:toyF} satisfies the conditions of Theorem \ref{th:main}
with equilibrium $X^*(\nu,\eta) = \bO$,
and discontinuity surface
$\Sigma_{\nu,\eta} = \left\{ X \in \mathbb{R}^3 \,\middle|\, X_1 = \nu - 2 X_2 - 3 X_3 \right\}$.

The boundary equilibrium bifurcation condition \eqref{eq:aBEB} holds
because $X^*(\nu,\eta) \in \Sigma_{\nu,\eta}$ exactly when $\nu = 0$.
The Hopf bifurcation condition \eqref{eq:aHB} holds
because $X^*(\nu,\eta)$ has purely imaginary eigenvalues exactly when $\eta = 0$.
The quantities \eqref{eq:udMR} have the values
\begin{align}
u &= \begin{bmatrix} -1 \\ -2 \\ -3 \end{bmatrix}, &
d &= \begin{bmatrix} 0 \\ -2 \\ 1 \end{bmatrix}, &
M_R &= \begin{bmatrix} 0 & 1 & 0 \\ -1 & 0 & 0 \\ 0 & 0 & -\frac{1}{5} \end{bmatrix},
\end{align}
so $\beta_0 = 1$ and $\gamma_0 = -\frac{1}{5}$
because $M_R$ has eigenvalues $\pm \ri$ and $-\frac{1}{5}$.
Normalised left and right eigenvectors for the eigenvalue $-\frac{1}{5}$
are $w = v = e_3$ (the third standard basis vector of $\mathbb{R}^3$).
Thus $u$ and $w$ are linearly independent, i.e.~condition \eqref{eq:auw} holds.
Notice $u^{\sf T} d = 1 > 0$,
so the system has an attracting region.
Also $u^{\sf T} M_R^{-1} d = 13 > 0$,
and the transversality condition \eqref{eq:anu} holds,
thus for sufficiently small values of $\eta$
the boundary equilibrium bifurcations along $\nu = 0$ are nonsmooth-folds.
Thus for $\nu > 0$, the equilibrium coexists with a pseudo-equilibrium.
Both equilibria are visible in the phase portrait, Fig.~\ref{fig:toy1}(i).

\begin{figure}[b!]
\begin{center}
\includegraphics[width=15.6cm]{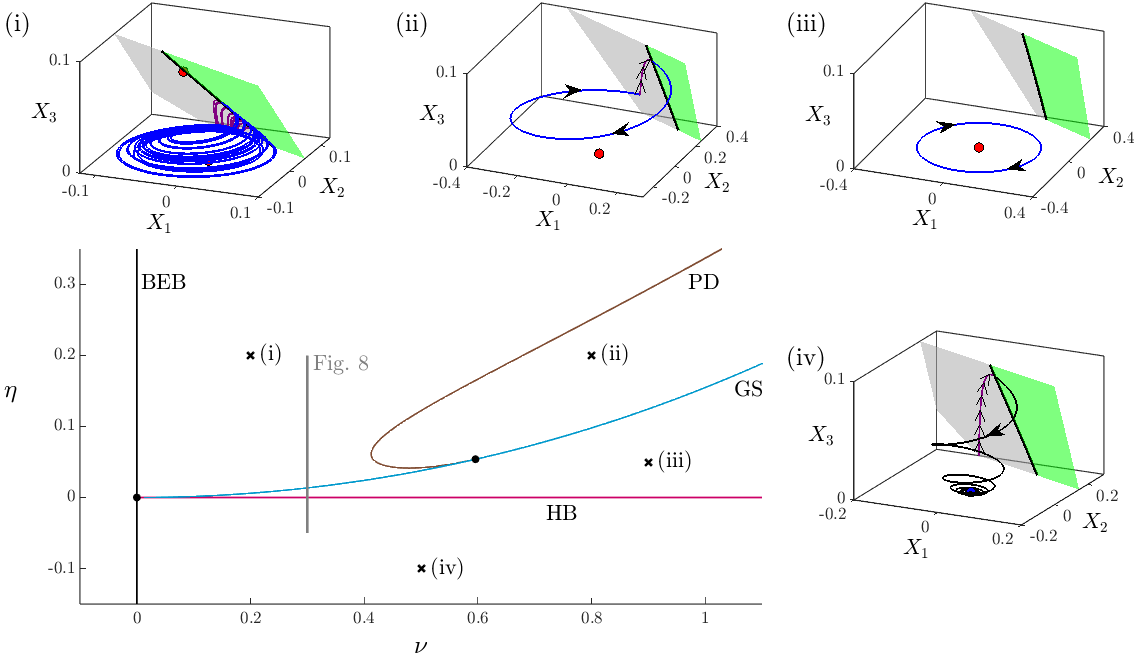}
\caption{
A two-parameter bifurcation diagram and representative phase portraits of \eqref{eq:toyF}
(BEB: nonsmooth-fold boundary equilibrium bifurcation;
HB: supercritical Hopf bifurcation;
GS: grazing-sliding bifurcation;
PD: period-doubling bifurcation of a limit cycle with one sliding segment).
The phase portraits use the following parameter values:
(i) $(\nu,\eta) = (0.2,0.2)$;
(ii) $(\nu,\eta) = (0.8,0.2)$;
(iii) $(\nu,\eta) = (0.9,0.05)$;
(iv) $(\nu,\eta) = (0.5,-0.1)$.
In the each phase portrait the equilibrium (circle) at the origin is coloured blue
if it is stable and red if it is unstable.
The unstable pseudo-equilibrium (red circle) is only visible in panel (i).
\label{fig:toy1}
} 
\end{center}
\end{figure}

The variables $X_1$ and $X_2$ of the right piece $\dot{X} = F_R(X;\nu,\eta)$
are decoupled from $X_3$ and evolve according to the van der Pol system \cite{Va26}.
The Hopf bifurcation of this system satisfies
the transversality condition \eqref{eq:aeta} and has $\chi_{\rm HB} < 0$, see \cite{Si22b},
so the bifurcation generates a stable limit cycle.
This limit cycle grazes the discontinuity surface
along the curve GS in Fig.~\ref{fig:toy1}, which was computed by numerical continuation.

\begin{figure}[b!]
\begin{center}
\includegraphics[width=15.6cm]{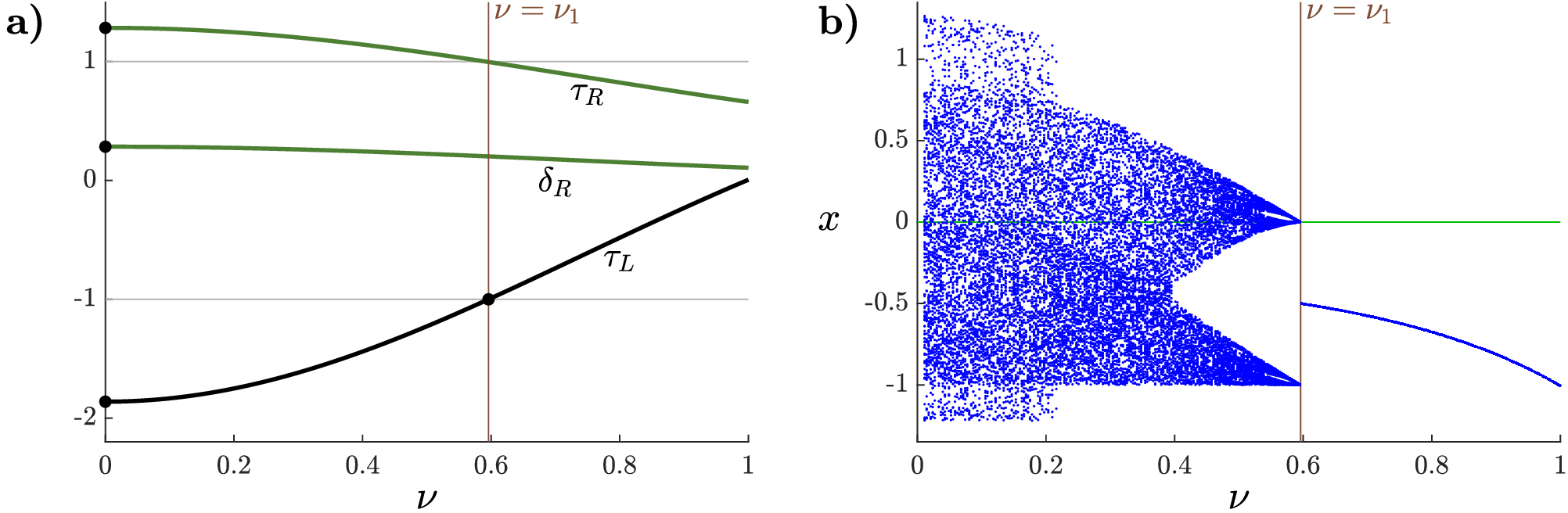}
\caption{
Panel (a) shows the values of
$\tau_L$, $\tau_R$, and $\delta_R$ for the grazing-sliding bifurcation curve GS of Fig.~\ref{fig:toy1}.
The black dots at $\nu = 0$ indicate the values \eqref{eq:tL0}--\eqref{eq:dR0}.
Panel (b) is a bifurcation diagram of
the two-dimensional border-collision normal form \eqref{eq:bcnf}
with $\mu = -1$, $\delta_L = 0$, and the values in panel (a)
(for each value of $\nu$ we computed $10^4$ iterates of $(x,y) = (0,0)$
and plotted the last $100$ values of $x$).
\label{fig:toy2}
} 
\end{center}
\end{figure}

The local dynamics created along this curve
are captured by the two-dimensional border-collision normal form \eqref{eq:bcnf}.
Fig.~\ref{fig:toy2}(a) shows how the parameter values of \eqref{eq:bcnf}
vary along the curve (also $\delta_L = 0$).
These values were obtained by using finite difference approximations
to evaluate the derivative of the global Poincar\'e map $Q_{\rm global}$,
and by using the formula \eqref{eq:8p74} to evaluate the derivative of the discontinuity map $Q_{{\rm disc}, L}$.
The parameter values of \eqref{eq:bcnf} are then taken
to be the traces and determinants of
$\rD Q_{\rm global}(Q_{{\rm disc},L}) \rD Q_{{\rm disc},L}$ and $\rD Q_{\rm global}$.

The black dots in Fig.~\ref{fig:toy2}(a) at $\nu = 0$ are the values
given the formulas \eqref{eq:tL0}--\eqref{eq:dR0}, specifically
$\tau_L = -1.8616$, $\tau_R = 1.2846$, and $\delta_R = 0.2846$, to four decimal places.
This corresponds to a point in region 3 of Fig.~\ref{fig:mapBifSetA}.
Therefore, above the curve GS, and with sufficiently small $\nu > 0$,
we expect \eqref{eq:toyF} has a chaotic attractor, as appears to the case
at parameter point (i) of Fig.~\ref{fig:toy1}.

\begin{figure}[b!]
\begin{center}
\includegraphics[width=9cm]{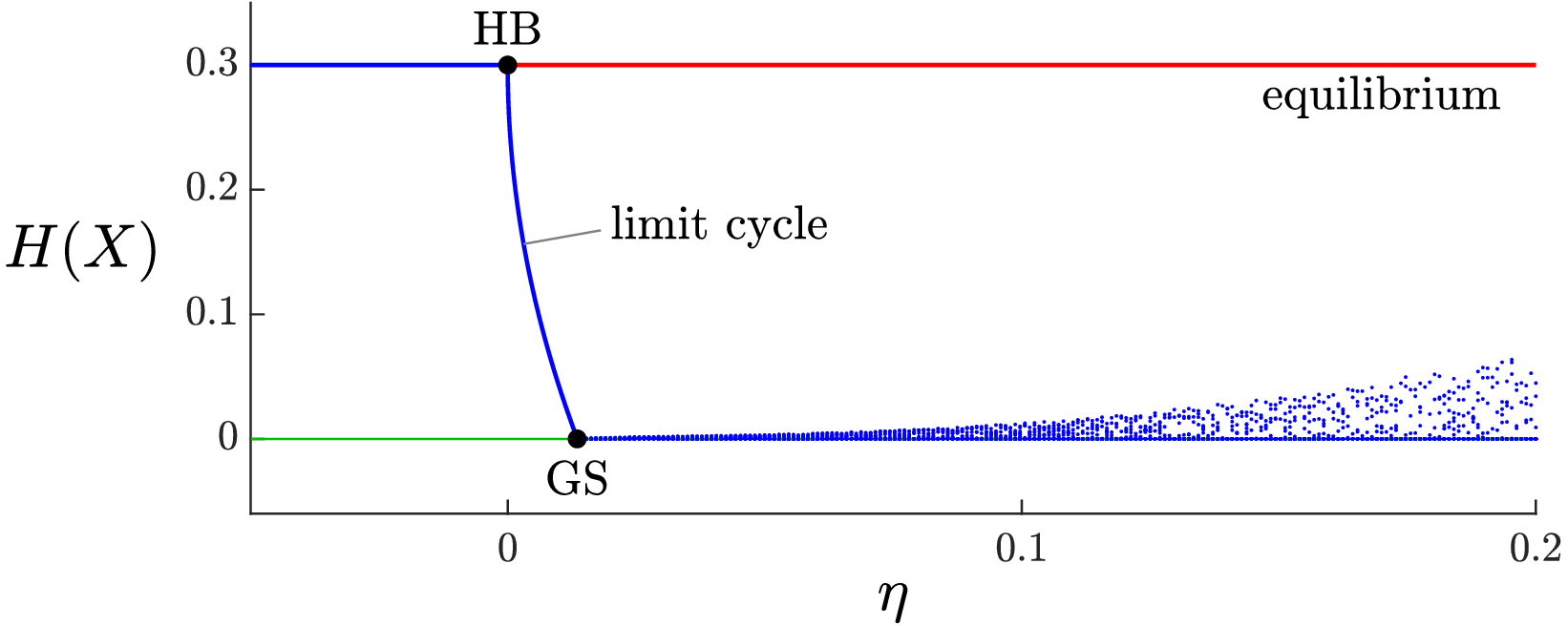}
\caption{
A one-parameter bifurcation diagram of \eqref{eq:toyF} with $\nu = 0.3$
corresponding to the grey line segment in Fig.~\ref{fig:toy1}
(HB: supercritical Hopf bifurcation; GS: grazing-sliding bifurcation).
For the limit cycle and presumably chaotic attractor,
we plot $H$-values at points where these objects intersect the surface
$\Omega = \left\{ X \in \mathbb{R}^3 \,\middle|\, \cL_{F_R} H(X) = 0 \right\}$.
\label{fig:toy3}
} 
\end{center}
\end{figure}

Fig.~\ref{fig:toy3} is a one-parameter bifurcation diagram
showing how the attractor changes as we pass through the Hopf
and grazing-sliding bifurcations when $\nu = 0.3$.
The presumably chaotic attractor created at grazing-sliding
emerges softly out of the grazing limit cycle,
then grows steadily in size as $\eta$ increases further.

To understand the behaviour of the system for larger values of $\nu$,
we show in Fig.~\ref{fig:toy2}(b) a numerically computed bifurcation diagram
of the border-collision normal form with $\mu = -1$ and the parameter values of Fig.~\ref{fig:toy2}(a).
At any typical $\nu$-value, the attractor in Fig.~\ref{fig:toy2}(b)
corresponds to the attractor of \eqref{eq:toyF} created in the grazing-sliding bifurcation.
We conclude that the grazing-sliding bifurcation creates a chaotic attractor
until $\nu_1 \approx 0.5959$, where $\tau_L = -1$.
Beyond this value the grazing-sliding bifurcation creates a stable limit cycle with one sliding segment,
as in Fig.~\ref{fig:toy1}(ii).
With further parameter variation, this limit cycle
loses stability along the period-doubling curve PD
that emanates from the grazing-sliding curve at $\nu = \nu_1$.
There are additional bifurcation curves that we have not computed,
such as further period-doubling bifurcations between parameter points (ii) and (i),
and curves where the topology of the chaotic attractor changes between $\nu = 0$ and $\nu = \nu_1$.

\subsection{Pest control}

Consider a system of the form
\begin{equation}
\dot{X} = F_R(X) =
\begin{bmatrix}
X_1 (1-X_1) - f_1(X_1) X_2 \\
f_1(X_1) X_2 - f_2(X_2) X_3 - d_1 X_2 \\
f_2(X_2) X_3 - d_2 X_3
\end{bmatrix},
\label{eq:appl-FR}
\end{equation}
where $X_1$, $X_2$, and $X_3$ represent the populations of three different species.
This is a general form for a three-species food chain model,
where species $X_3$ feeds on species $X_2$, and species $X_2$ feeds on species $X_1$.
The constants $d_1 > 0$ and $d_2 > 0$ are the death rates of $X_2$ and $X_3$,
and the carrying capacity of $X_1$ has been scaled to $1$.
We assume the functional responses $f_1$ and $f_2$ have the Hollings Type II form
\begin{align}
f_1(X_1) &= \frac{a_1 X_1}{1 + b_1 X_1}, &
f_2(X_2) &= \frac{a_2 X_2}{1 + b_2 X_2}.
\nonumber
\end{align}
With these functions, the system \eqref{eq:appl-FR} was studied by Hastings and Powell \cite{HaPo91}
who discovered that the populations of the species could settle to a chaotic attractor.
Other research groups have studied the system further 
finding Hopf bifurcations, homoclinic bifurcations,
and other dynamical phenomena \cite{AbRo94,KuRi96,McYo95}.

If the three-species system is subject to human management or control,
such as the release of pesticides,
we assume that the populations instead evolve according to $\dot{X} = F_L(X)$, where
\begin{equation}
F_L(X) = F_R(X) - \begin{bmatrix} q_1 X_1 \\ q_2 X_2 \\ q_3 X_3 \end{bmatrix},
\label{eq:appl-FL}
\end{equation}
and $q_1$, $q_2$, and $q_3$ are the killing rates of $X_1$, $X_2$, and $X_3$, respectively.
To reduce costs and lessen detrimental effects such as harm to non-invasive species,
such control is only applied intermittently \cite{Ko98}.
For example, the control may only be applied
when the population of one of the species exceeds a threshold \cite{Jo04}.
If we ignore the time lag between when populations are measured and control is applied,
which may be reasonable if the natural time-scale of the ecological dynamics
is significantly slower than the human response,
the system may be treated as a Filippov system.

Zhou and Tang \cite{ZhTa22}
treat $X_1$, $X_2$, and $X_3$ as the populations of a crop,
a pest, and an enemy of the pest, respectively.
They suppose control is applied when the population of the pest enemy falls
below a threshold $\xi$, as in this case the pest population is in danger of spiking.
That is, the model is a Filippov system \eqref{eq:F},
where $F_L$ and $F_R$ are given by \eqref{eq:appl-FL} and \eqref{eq:appl-FR}, and
\begin{equation}
H(X) = X_3 - \xi.
\label{eq:HZhTa22}
\end{equation}
Here we analyse a boundary Hopf bifurcation in this model
using $b_1$ and $\xi$ as bifurcation parameters.
Following Hastings and Powell \cite{HaPo91}, we fix
\begin{align}
a_1 &= 5, &
a_2 &= 0.1, &
b_2 &= 2, &
d_1 &= 0.4, &
d_2 &= 0.01,
\label{eq:appl-paramBasic}
\end{align}
and as in Zhou and Tang \cite{ZhTa22} we use
\begin{align}
q_1 &= 0, &
q_2 &= 0.05, &
q_3 &= -0.01,
\label{eq:appl-paramZhTa22}
\end{align}
where $q_3 < 0$ because enemies of the pest
are released into the environment during the control phase.

\begin{figure}[b!]
\begin{center}
\includegraphics[width=15.6cm]{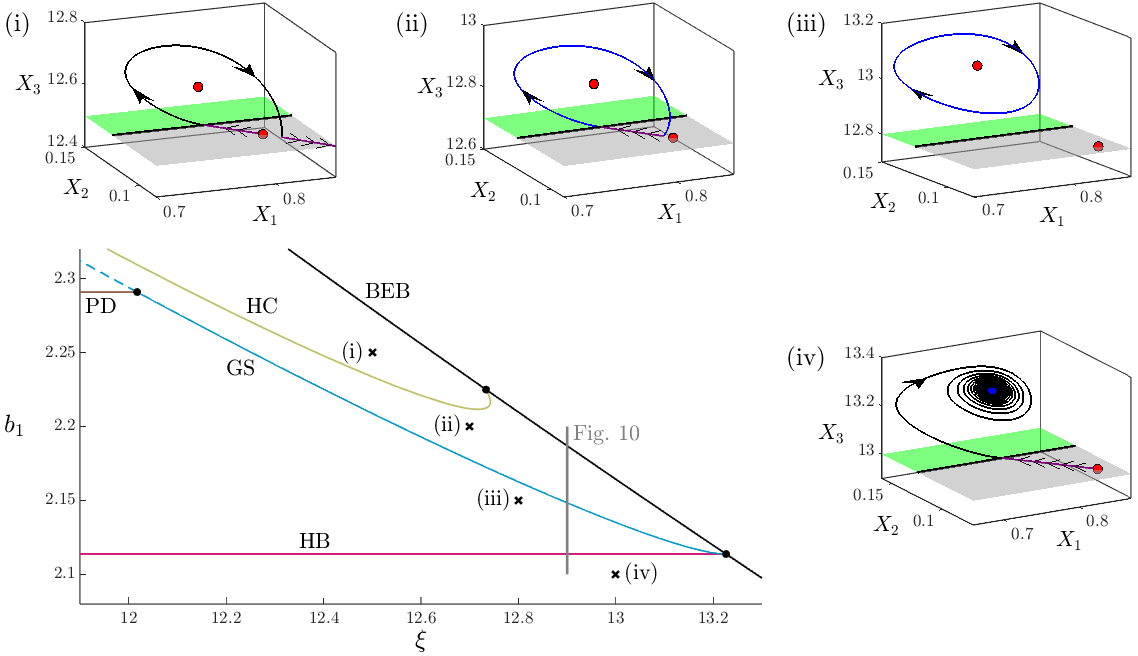}
\caption{
A two-parameter bifurcation diagram and representative phase portraits
of the food chain model with threshold control on the value of $X_3$,
specifically \eqref{eq:F} with \eqref{eq:appl-FR}--\eqref{eq:appl-paramZhTa22}
(BEB: nonsmooth-fold boundary equilibrium bifurcation;
HB: supercritical Hopf bifurcation;
PD: period-doubling bifurcation;
GS: grazing-sliding bifurcation;
HC: homoclinic bifurcation of sliding limit cycle and pseudo-equilibrium).
The curve GS is solid where the grazing limit cycle is stable,
and dashed where it is unstable.
The phase portraits use the same conventions as Fig.~\ref{fig:toy1}
and the parameter values:
(i) $(\xi,b_1) = (12.5,2.25)$;
(ii) $(\xi,b_1) = (12.7,2.2)$;
(iii) $(\xi,b_1) = (12.8,2.15)$;
(iv) $(\xi,b_1) = (13,2.1)$.
\label{fig:ecol5}
} 
\end{center}
\end{figure}

With a relatively small value of $b_1$,
the system without control has a stable equilibrium $X^*(b_1)$, Fig.~\ref{fig:ecol5}(iv).
As the value of $b_1$ is increased, $X^*(b_1)$ loses stability in
a supercritical Hopf bifurcation at $b_1 = b_{1,{\rm HB}} \approx 2.1138$,
creating a stable limit cycle, Fig.~\ref{fig:ecol5}(iii).
As $b_1$ is increased further,
the limit cycle loses stability in a period-doubling bifurcation
at $b_1 = b_{1,{\rm PD}} \approx 2.2909$.
The limit cycle undergoes grazing
along the curve GS in Fig.~\ref{fig:ecol4}.
Also, the equilibrium $X^*(b_1)$
hits the threshold $X_3 = \xi$ along the curve BEB.

The intersection of the Hopf and boundary equilibrium bifurcation curves is a boundary Hopf bifurcation.
This point is located at $(\xi,b_1) = (\xi_3^*,b_{1,{\rm HB}})$,
where $\xi_3^* \approx 13.23$ is the $X_3$-value of the equilibrium at the Hopf bifurcation.
A numerical evaluation of the limits \eqref{eq:tL0}--\eqref{eq:dR0} yields the values
$\tau_L \approx -0.02301$, $\tau_R \approx 1$, and $\delta_R \approx 0$.
As indicated in Fig.~\ref{fig:mapBifSetA},
at these values the border-collision normal form with $\mu = -1$ and $\delta_L = 0$
has a stable fixed point with $x < 0$.
Since the fixed point corresponds to a limit cycle with a sliding segment, as in Fig.~\ref{fig:ecol5}(ii),
we can conclude that as we pass through the grazing-sliding bifurcation curve
at any point sufficiently close to the boundary Hopf bifurcation,
the limit cycle remains stable but gains a sliding segment.

\begin{figure}[b!]
\begin{center}
\includegraphics[width=8cm]{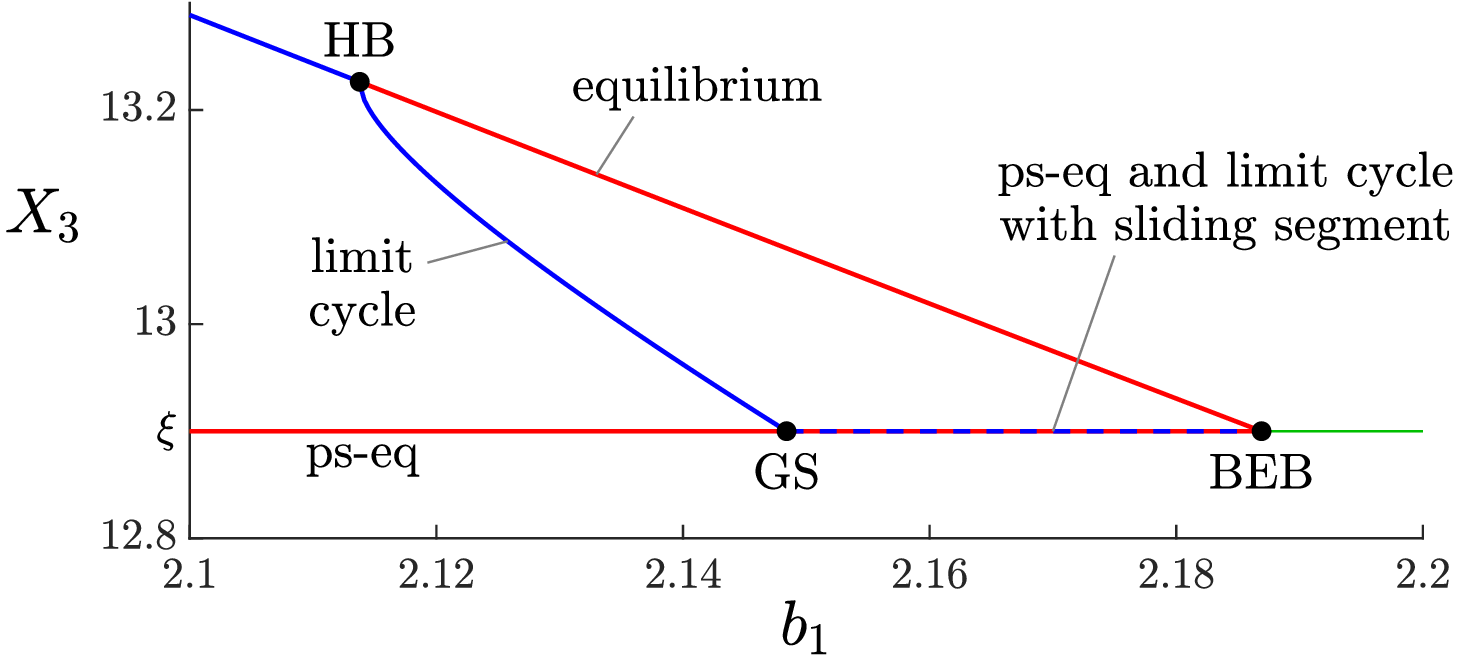}
\caption{
A one-parameter bifurcation diagram of the food chain model with
threshold control on the value of $X_3$ using
the parameter values of Fig.~\ref{fig:ecol5} and $\xi = 12.9$.
\label{fig:ecol7}
} 
\end{center}
\end{figure}

This is seen in the bifurcation diagram, Fig.~\ref{fig:ecol7}.
The limit cycle is born in the Hopf bifurcation,
accrues a sliding segment at the grazing-sliding bifurcation,
then collaspes to a point and is destroyed at the boundary equilibrium bifurcation.

Also, $|\tau_L| < 1$ along the grazing-sliding curve,
thus the system has a stable limit cycle with a sliding segment
immediately above the entirety of grazing-sliding curve.
At larger values of $b_1$,
the limit cycle with a sliding segment is destroyed in a homoclinic
bifurcation where it collides with the pseudo-equilibrium
associated with the boundary equilibrium bifurcation.
Beyond the period-doubling bifurcation the limit cycle is unstable when it undergoes grazing.
It remains for future work to compute the bifurcation curve along which the
period-doubled solution undergoes grazing,
and analogous curves for subsequent period-doubling bifurcations.

In summary, the system without control
has a stable equilibrium at small values of $b_1$,
and a stable limit cycle at intermediate values of $b_1$.
With threshold control, the system has a grazing-sliding bifurcation where
the limit cycle gains a sliding segment corresponding
to the control being applied intermittently.
The limit cycle with a sliding segment is subsequently
destroyed in either a homoclinic bifurcation
or a boundary equilibrium bifurcation.
Beyond these bifurcations typical orbits converge
to a steady-state solution at which the pests have been eliminated ($X_2 = 0$).

\subsection{Population dynamics with a harvesting threshold}

Hamdallah {\em et al.}~\cite{HaAr21} study the same system
but treat $X_1$, $X_2$, and $X_3$ as the populations
of a prey, a middle predator, and a top predator, respectively,
and suppose harvesting is permitted when the prey population exceeds a threshold $\xi$, i.e.
\begin{equation}
H(X) = \xi - X_1 \,.
\label{eq:HHaAr21}
\end{equation}
Here we study this system using again the parameters \eqref{eq:appl-paramBasic}, but now
\begin{align}
q_1 &= 0.09, &
q_2 &= 0.01, &
q_3 &= 0.001,
\label{eq:appl-paramHaAr21}
\end{align}
as in \cite{HaAr21}.
The Hopf and period-doubling bifurcation curves are unchanged,
because these operate independently of the control,
but the system has different boundary equilibrium and grazing-sliding bifurcation curves,
see Fig.~\ref{fig:ecol1}.

\begin{figure}[b!]
\begin{center}
\includegraphics[width=15.6cm]{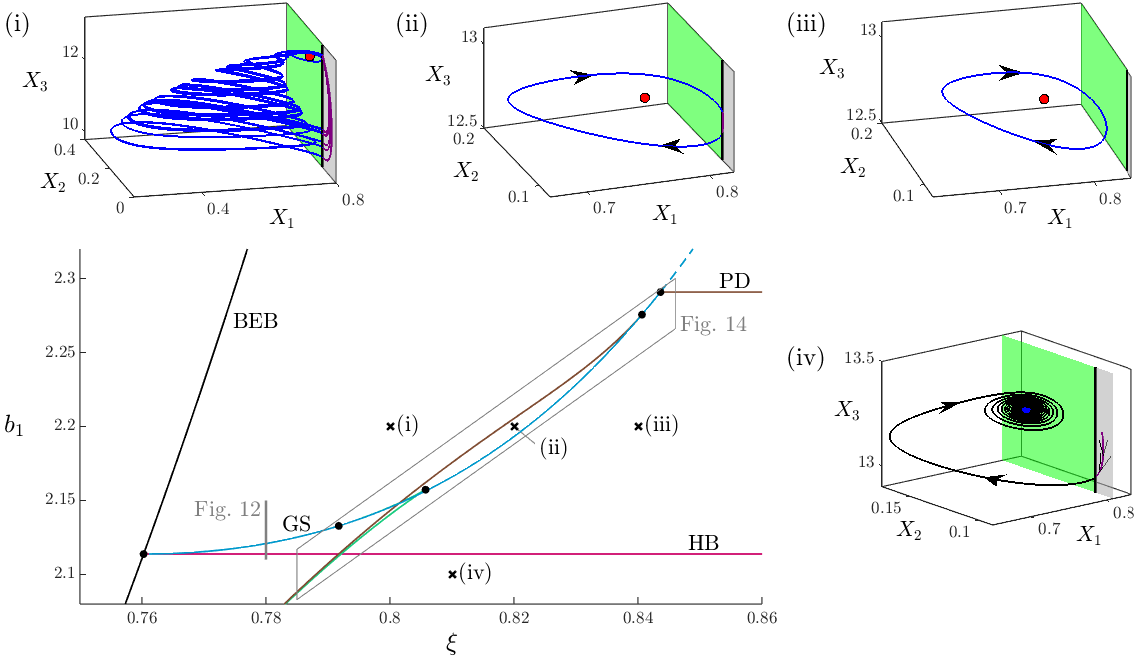}
\caption{
A two-parameter bifurcation diagram and representative phase portraits
of the food chain model with threshold control on the value of $X_1$,
specifically \eqref{eq:F} with \eqref{eq:appl-FR}, \eqref{eq:appl-FL},
\eqref{eq:appl-paramBasic}, \eqref{eq:HHaAr21}, and \eqref{eq:appl-paramHaAr21}
(BEB: nonsmooth-fold boundary equilibrium bifurcation;
HB: supercritical Hopf bifurcation;
PD: period-doubling bifurcation;
GS: grazing-sliding bifurcation).
The curve GS is solid where the grazing limit cycle is stable,
and dashed where it is unstable.
The phase portraits use the same conventions as Fig.~\ref{fig:toy1}
and the parameter values:
(i) $(\xi,b_1) = (0.8,2.2)$;
(ii) $(\xi,b_1) = (0.82,2.2)$;
(iii) $(\xi,b_1) = (0.84,2.2)$;
(iv) $(\xi,b_1) = (0.81,2.1)$.
\label{fig:ecol1}
} 
\end{center}
\end{figure}

The Hopf and boundary equilibrium bifurcation curves 
intersect at $(\xi,b_1) = (\xi_1^*,b_{1,{\rm HB}})$,
where $\xi_1^* \approx 0.7603$ is the $X_1$-value
of the equilibrium of \eqref{eq:appl-FR} at the Hopf bifurcation.
The limiting values \eqref{eq:tL0}--\eqref{eq:dR0}
are $\tau_L \approx 1.541$, $\tau_R \approx 1$, and $\delta_R \approx 0$,
which corresponds to a point in Fig.~\ref{fig:mapBifSetA} where
the border-collision normal form has no attractor.
Thus, near the boundary Hopf bifurcation,
the grazing-sliding bifurcation generates no local attractor.
Here the stable limit cycle is destroyed
and typical orbits are ejected into a different area of phase space
where they often converge to an apparently chaotic attractor, as in Fig.~\ref{fig:ecol1}(i).

\begin{figure}[b!]
\begin{center}
\includegraphics[width=8cm]{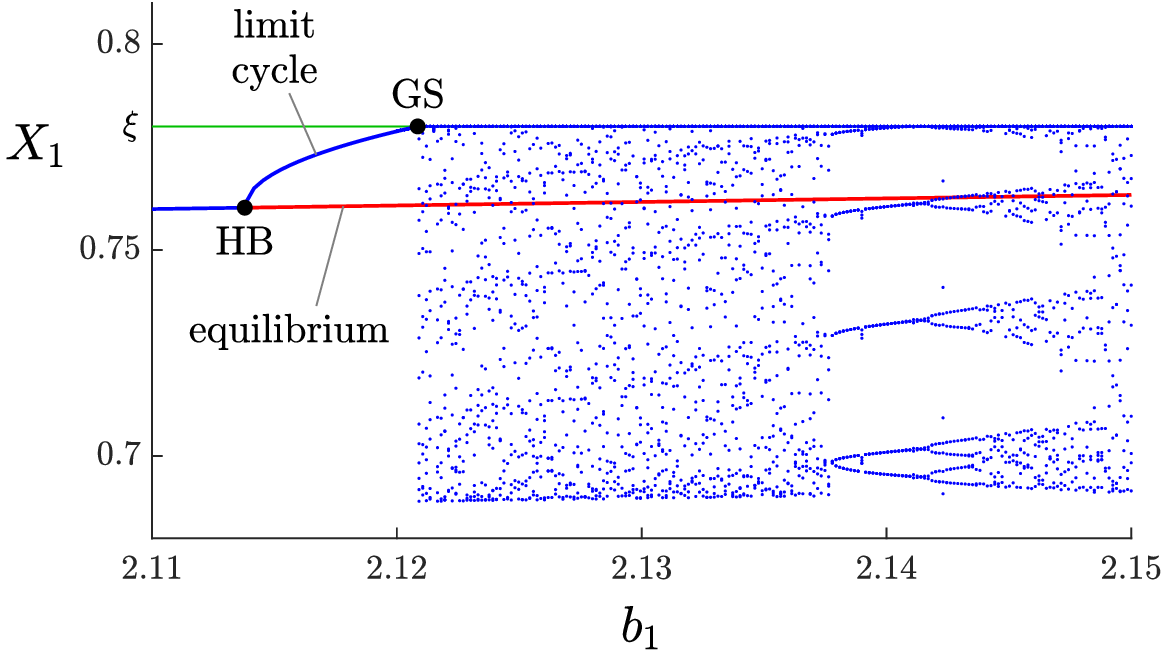}
\caption{
A one-parameter bifurcation diagram of the food chain model with
threshold control on the value of $X_1$
using the parameter values of Fig.~\ref{fig:ecol1} and $\xi = 0.78$.
\label{fig:ecol4}
} 
\end{center}
\end{figure}

Fig.~\ref{fig:ecol4} shows a typical example of this transition.
At the grazing-sliding bifurcation the attractor
jumps from a limit cycle to a large amplitude chaotic attractor.
This is distinct from the transition in
Fig.~\ref{fig:toy3} where the chaotic attractor grows softly out of the limit cycle.

\begin{figure}[b!]
\begin{center}
\includegraphics[width=7.15cm]{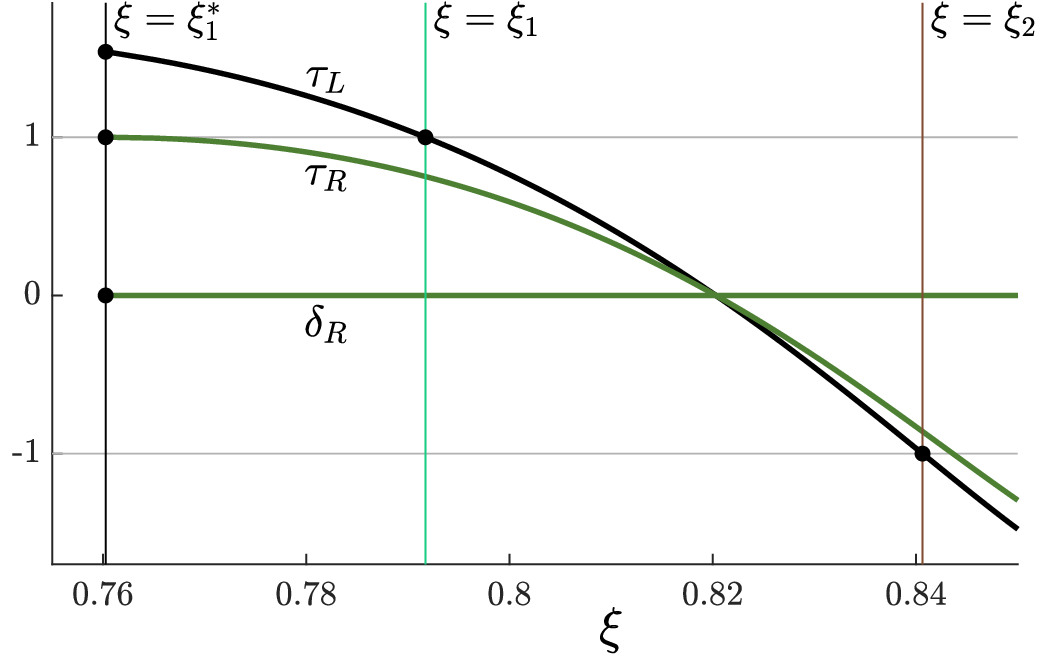}
\caption{
Numerically computed parameter values of
$\tau_L$, $\tau_R$, and $\delta_R$ for the grazing-sliding bifurcation curve of Fig.~\ref{fig:ecol1}.
The black dots at $\xi = \xi_1^* \approx 0.7603$ indicate the values \eqref{eq:tL0}--\eqref{eq:dR0}.
\label{fig:ecol3}
} 
\end{center}
\end{figure}

For the grazing-sliding bifurcation curve of Fig.~\ref{fig:ecol1},
the parameter values of the corresponding border-collision normal form
are plotted in Fig.~\ref{fig:ecol3} (also $\delta_L = 0$).
This reveals the occurrence of codimension-two points at $\xi = \xi_1 \approx 0.7917$
and $\xi = \xi_2 \approx 0.8406$ where $\tau_L = 1$ and $\tau_L = -1$, respectively.
Between these points the border-collision normal form with $\mu = -1$
has a stable fixed point with $x < 0$.
Thus, between these points,
the stable limit cycle gains a sliding segment at the grazing-sliding bifurcation.

\begin{figure}[b!]
\begin{center}
\includegraphics[width=12cm]{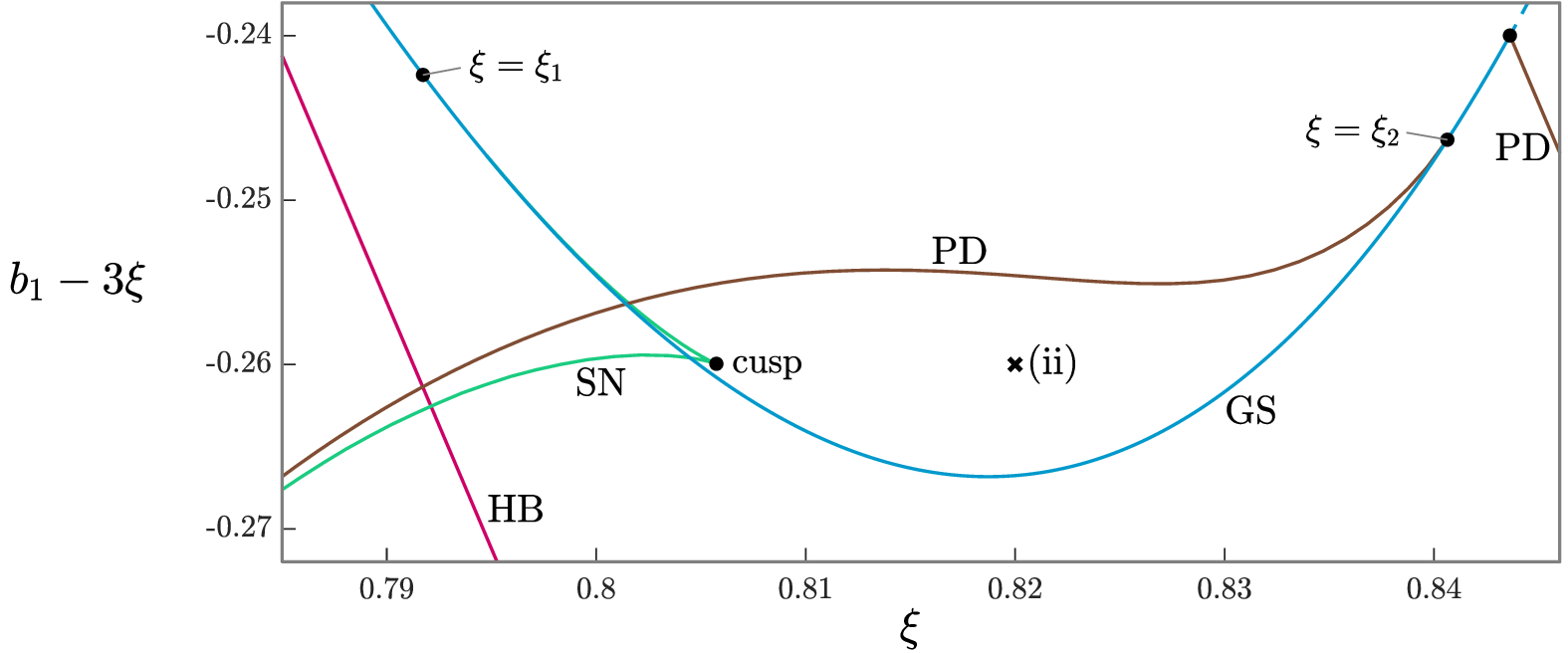}
\caption{
A magnified view of Fig.~\ref{fig:ecol1}.
\label{fig:ecol2}
} 
\end{center}
\end{figure}

From the codimension-two points at $\xi = \xi_1$ and $\xi = \xi_2$,
there issue curves of saddle-node and period-doubling bifurcations.
These are seen more clearly in the magnification, Fig.~\ref{fig:ecol2}.
The saddle-node curve is initially very close to the grazing-sliding curve
because these curves are cubically tangent at $\xi = \xi_1$
by a result of Nordmark and Kowalcyzk \cite{NoKo06}.
This curve experiences a cusp bifurcation at $\xi \approx 0.8057$ and abruptly changes direction.

In summary, the addition of threshold control,
but now on the value of $X_1$,
has introduced a curve of grazing-sliding bifurcations
at which the stable limit cycle either jumps to a large amplitude chaotic attractor,
or gains a sliding segment.
The limit cycle with a sliding segment loses stability along a curve of period-doubling bifurcations.
It remains to identify further bifurcations
to explain how the period-doubled solution transitions to a chaotic attractor.

\section{Summary and discussion}
\label{sec:conc}

We have shown that boundary Hopf bifurcations in three-dimensional Filippov systems
correspond to pairs $(\tau_L,\tau_R)$
that govern the nature of the attractor created along the curve of grazing-sliding bifurcations.
Three examples were explored in \S\ref{sec:appl}.
The first example has $(\tau_L,\tau_R) \approx (-1.86,1.28)$,
corresponding to a chaotic attractor with motion near the grazing limit cycle, Fig.~\ref{fig:toy1}(i).
The second example has $(\tau_L,\tau_R) \approx (-0.02,1)$,
corresponding to a stable limit cycle with a segment of sliding motion, Fig.~\ref{fig:ecol5}(ii).
The third example has $(\tau_L,\tau_R) \approx (1.541,1)$,
corresponding to the absence of an attractor.
In this case, as we pass through the grazing-sliding bifurcation
typical orbits near the grazing limit cycle are ejected to a different area of phase space
and converge to a chaotic attractor far from the grazing limit cycle, Fig.~\ref{fig:ecol1}(i).

The second and third examples are models of a three-species food chain with threshold control.
For these models we have fixed all but two of the parameter values
and analysed the dynamics near boundary Hopf bifurcations.
For a broader bifurcation analysis of the models
refer to the original papers \cite{HaAr21,ZhTa22}.

There are many avenues for future work.
For instance, it should be possible to perform a similar characterisation of the dynamics
created along the boundary equilibrium bifurcation curve.
In generic situations these dynamics are captured by the three-dimensional
boundary equilibrium bifurcation normal form which has five parameters
in addition to a parameter $\mu$ that can be scaled to $\pm 1$ \cite{Si18d}.
What subfamily of this normal form is relevant near boundary Hopf bifurcations?

It remains to construct rigorous proofs for the division of the $(\tau_L,\tau_R)$-plane
indicated in Figs.~\ref{fig:mapBifSetA} and \ref{fig:mapBifSetB}.
We expect this can be done by combining global arguments
with calculations for the amount by which line segments
expand when mapped under the normal form.
If it can be shown that the normal form is {\em locally eventually onto}
on a disjoint union of line segments $\mathcal{A}$,
then the normal form is chaotic in sense of Devaney on $\mathcal{A}$,
and no subset of $\mathcal{A}$ is an attractor \cite{GlJe19,Ru17}.
Already results of this type have been obtained by Kowalcyzk \cite{Ko05}.

Other codimension-two phenomena involving grazing-sliding bifurcations remain to be unfolded,
such as cascades of grazing-sliding bifurcations
resulting from boundary equilibrium bifurcations
that involve homoclinic or heteroclinic connections \cite[pg.~373]{DiBu08}.

As a final remark, our derivation in Appendix \ref{app:formula}
of the linear term of the discontinuity map associated with grazing-sliding bifurcations
involved the computation of a {\em virtual counterpart}, which we believe is a new methodology.
By computing the differences between the virtual counterpart
and the start and end points of the discontinuity map,
rather than computing these points directly,
it was sufficient to retain only the leading-order terms at every step in our asymptotic calculations.
We expect that this methodology can greatly reduce the complexity of
discontinuity map computations for other sliding bifurcations,
especially adding-sliding bifurcations for which
a brute-force approach requires retaining terms to fourth order,
only to have all terms at first, second, and third order
vanish in the final result \cite{DiBu08,DiKo02}.

\section*{Acknowledgements}

This work was supported by Marsden Fund contract MAU2209 managed by Royal Society Te Ap\={a}rangi.
The author thanks Isaac Abbott for assistance with the numerical explorations in \S\ref{sec:bcnf}.

\appendix

\section{The derivative of the global map}
\label{app:global}

Here we derive the formula \eqref{eq:DQglobal} for the derivative of $Q_{\rm global}$
at the grazing point and in the limit $\nu \to 0$.
We consider the initial point \eqref{eq:perturbedPoint},
and evolve under $\dot{Y} = \tilde{F}_R(Y;0)$ under returning to $\Omega_0$.
By \eqref{eq:tildeFR2}, $\tilde{F}_R(Y;0) = \tilde{M}_R Y$,
where $\tilde{M}_R$ is given by \eqref{eq:realJordanForm}.
Thus the orbit is given explicitly by
\begin{equation}
\varphi_t(Y) = \begin{bmatrix}
\cos(\beta_0 t) Y_1 + \sin(\beta_0 t) Y_2 \\
-\sin(\beta_0 t) Y_1 + \cos(\beta_0 t) Y_2 \\
\re^{\gamma_0 t} Y_3
\end{bmatrix},
\label{eq:flow}
\end{equation}
where
\begin{align}
Y_1 &= \frac{-a \psi}{a^2 + b^2} + \ee_1 \,, &
Y_2 &= \frac{-b \psi}{a^2 + b^2} + \frac{b}{a} \,\ee_1 - \frac{c \gamma_0}{a \beta_0}, &
Y_3 &= \ee_3 \,.
\label{eq:perturbedPoint2}
\end{align}
If $\ee_1 = \ee_3 = 0$,
the orbit is periodic and returns to $\Omega_0$ after a time $\frac{2 \pi}{\beta_0}$.
Thus for small $\ee_1, \ee_3 \in \mathbb{R}$ the evolution time associated with $Q_{\rm global}$ is
\begin{equation}
T_{\rm evol}(\ee_1,\ee_3) = \frac{2 \pi}{\beta_0} + k_1 \ee_1 + k_3 \ee_3
+ \cO \left( \left( |\ee_1| + |\ee_3| \right)^2 \right),
\label{eq:Tevol}
\end{equation}
for some $k_1, k_3 \in \mathbb{R}$.
By substituting \eqref{eq:perturbedPoint2} and \eqref{eq:Tevol} into \eqref{eq:flow}, we obtain
\begin{align}
Y_1' &= \frac{-a \psi}{a^2 + b^2} + \left( 1 - \frac{b \beta_0 \psi k_1}{a^2 + b^2} \right) \ee_1
- \frac{b \beta_0 \psi k_3}{a^2 + b^2} \,\ee_3 + \cO \left( \left( |\ee_1| + |\ee_3| \right)^2 \right),
\label{eq:Y1p} \\
Y_2' &= \frac{-b \psi}{a^2 + b^2} + \left( \frac{b}{a} + \frac{a \beta_0 \psi k_1}{a^2 + b^2} \right) \ee_1
+ \left( \frac{a \beta_0 \psi k_3}{a^2 + b^2}- \frac{c \gamma_0}{a \beta_0} \right) \ee_3
+ \cO \left( \left( |\ee_1| + |\ee_3| \right)^2 \right), \label{eq:Y2p} \\
Y_3' &= \re^{\frac{2 \pi \gamma_0}{\beta_0}} \ee_3 
+ \cO \left( \left( |\ee_1| + |\ee_3| \right)^2 \right), \label{eq:Y3p} 
\end{align}
using also $\cos \left( \beta_0 T_{\rm evol}(\ee_1,\ee_3) \right)
= 1 + \cO \left( \left( |\ee_1| + |\ee_3| \right)^2 \right)$ and
$\sin \left( \beta_0 T_{\rm evol}(\ee_1,\ee_3) \right)
= \beta_0 (k_1 \ee_1 + k_3 \ee_3) + \cO \left( \left( |\ee_1| + |\ee_3| \right)^2 \right)$.
The evolution time is such that $Y' \in \Omega_0$,
so by using \eqref{eq:tildevR2} to solve
$\cL_{\tilde{F}_R} \tilde{H}(Y';0) = 0$ we obtain
\begin{align}
k_1 &= 0, &
k_3 &= \frac{c \gamma_0}{\beta_0^2 \psi} \left( 1 - \re^{\frac{2 \pi \gamma_0}{\beta_0}} \right).
\label{eq:k1k3}
\end{align}
By then substituting \eqref{eq:k1k3} into \eqref{eq:Y1p} and \eqref{eq:Y3p},
and extracting the coefficients of the $\ee_1$ and $\ee_3$ terms, we arrive at the desired formula
\begin{equation}
\begin{bmatrix} \frac{\partial Y_1'}{\ee_1} & \frac{\partial Y_1'}{\ee_3} \\[1.2mm]
\frac{\partial Y_3'}{\ee_1} & \frac{\partial Y_3'}{\ee_3} \end{bmatrix}
= \begin{bmatrix} 1 & -\frac{b c \gamma_0}{\left( a^2 + b^2 \right) \beta_0}
\left( 1 - \re^{\frac{2 \pi \gamma_0}{\beta_0}} \right) \\
0 & \re^{\frac{2 \pi \gamma_0}{\beta_0}} \end{bmatrix}.
\nonumber
\end{equation}

\section{An explicit formula for the discontinuity map associated with a grazing-sliding bifurcation}
\label{app:formula}

Consider an $n$-dimensional Filippov system
\begin{equation}
\dot{X} = \begin{cases}
F_L(X), & H(X) < 0, \\
F_R(X), & H(X) > 0,
\end{cases}
\label{eq:gsF}
\end{equation}
with discontinuity surface $\Sigma = \left\{ X \in \mathbb{R}^n \,\middle|\, H(X) = 0 \right\}$,
and suppose $X_{\rm fold} \in \Sigma$ is a visible fold of $F_R$.
That is,
\begin{align}
H(X_{\rm fold}) &= 0, &
\cL_{F_R} H(X_{\rm fold}) &= 0, &
\cL_{F_R}^2 H(X_{\rm fold}) &> 0,
\label{eq:gsA1}
\end{align}
where the third quantity is the second Lie derivative of $H$ with respect to $F_R$.
Also suppose
\begin{equation}
\cL_{F_L} H(X_{\rm fold}) > 0, \\
\label{eq:gsA2}
\end{equation}
so that $\Sigma$ splits into a crossing region and an attracting sliding region in a
neighbourhood of $X_{\rm fold}$.
Consider the cross-section $\Omega = \left\{ X \in \mathbb{R}^n \,\middle|\, \cL_{F_R} H(X) = 0 \right\}$.
For any $X^{(1)} \in \Omega$ with $H \left( X^{(1)} \right) < 0$ sufficiently close to $X_{\rm fold}$,
let $X^{(2)} \in \Sigma$ be the result of following the flow of $F_R$ backwards until reaching $\Sigma$,
and let $X^{(3)} \in \Sigma \cap \Omega$ be the result of sliding forwards until reaching $\Omega$,
see Fig.~\ref{fig:schem2}.
The map $X^{(1)} \to X^{(3)}$ is the non-trivial part of the discontinuity map
associated with a grazing-sliding bifurcation in $n$-dimensional form \cite{DiBu08}.

\begin{figure}[b!]
\begin{center}
\includegraphics[width=5cm]{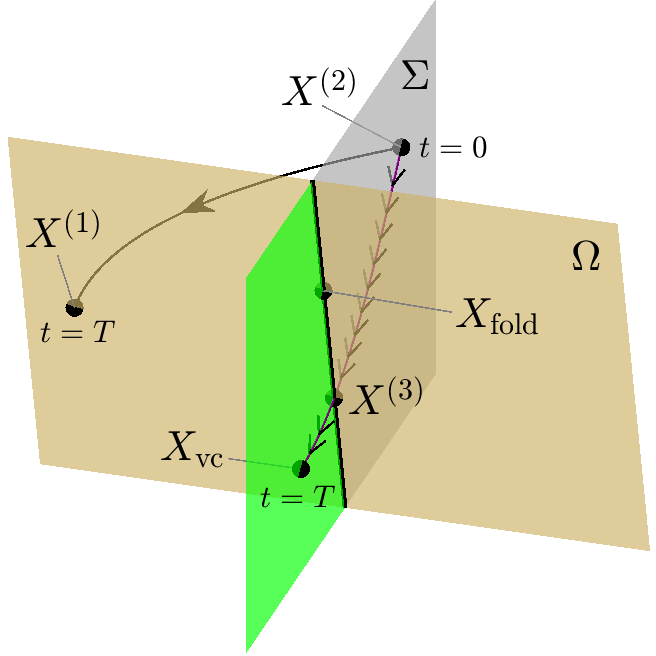}
\caption{
A sketch of the phase space of the $n$-dimensional system \eqref{eq:gsF}
subject to the assumptions in Proposition \ref{pr:gsFormula}.
The point $X^{(2)}$ belongs to $\Sigma$,
and by evolving this point forwards under $F_R$ for a time $T$,
we arrive at the point $X^{(1)} \in \Omega$
(the trajectory from $X^{(1)}$ to $X^{(2)}$ is {\em virtual} because it
is situated on the $H(X) < 0$ side of $\Sigma$).
By evolving $X^{(2)}$ under the sliding vector field for the same length of time $T$,
we arrive at the {\em virtual counterpart} $X_{\rm vc} \in \Sigma$.
\label{fig:schem2}
} 
\end{center}
\end{figure}

\begin{proposition}
If $F_L$ and $F_R$ are $C^2$, $H$ is $C^3$,
and \eqref{eq:gsA1} and \eqref{eq:gsA2} hold, then
\begin{equation}
X^{(3)} = X^{(1)} + \left( \frac{\cL_{F_L} \cL_{F_R} H}
{\big( \cL_{F_L} H \big) \big( \cL_{F_R}^2 H \big)} \,F_R - \frac{1}{\cL_{F_L} H} \,F_L \middle)
\right|_{X = X_{\rm fold}} H \left( X^{(1)} \right)
+ \cO \left( \left\| X^{(1)} - X_{\rm fold} \right\|^{\frac{3}{2}} \right).
\label{eq:gsFormula}
\end{equation}
\label{pr:gsFormula}
\end{proposition}

The coefficient of the $H \left( X^{(1)} \right)$ term
is a linear combination of $F_L$ and $F_R$.
It is in fact the unique linear combination
of $F_L$ and $F_R$ that puts the point $X^{(3)}$ on both $\Sigma$ and $\Omega$, at first order.

\begin{proof}
Without loss of generality we can assume $X_{\rm fold} = \bO$ (the origin).
Since $F_L$ and $F_R$ are $C^2$ and $H$ is $C^3$,
we have the following Taylor expansions about $X = \bO$:
\begin{align}
H(X) &= \nabla H(\bO)^{\sf T} X + \tfrac{1}{2} X^{\sf T} \rD^2 H(\bO) X
+ \cO \left( \| X \|^3 \right), \label{eq:gsH} \\
F_L(X) &= F_L(\bO) + \rD F_L(\bO) X + \cO \left( \| X \|^2 \right), \label{eq:gsFL} \\
F_R(X) &= F_R(\bO) + \rD F_R(\bO) X + \cO \left( \| X \|^2 \right), \label{eq:gsFR}
\end{align}
where we have substituted $H(\bO) = 0$.
We now Taylor expand the sliding vector field
$F_S(X) = \frac{F_R(X) \cL_{F_L} H(X) - F_L(X) \cL_{F_R} H(X)}{\cL_{F_L} H(X) - \cL_{F_R} H(X)}$.
Into this formula we insert
\begin{align}
\cL_{F_L} H(X) &= \cL_{F_L} H(\bO) + \left( \nabla \cL_{F_L} H(\bO) \right)^{\sf T} X + \cO \left( \| X \|^2 \right), \label{eq:gsLieL} \\
\cL_{F_R} H(X) &= \left( \nabla \cL_{F_R} H(\bO) \right)^{\sf T} X + \cO \left( \| X \|^2 \right), \label{eq:gsLieR}
\end{align}
where we have substituted $\cL_{F_R} H(\bO) = 0$, resulting in
\begin{equation}
F_S(X) = F_R(\bO) + \rD F_S(\bO) X + \cO \left( \| X \|^2 \right),
\label{eq:gsFS}
\end{equation}
where
\begin{equation}
\rD F_S(\bO) = \rD F_R(\bO)
+ \frac{(F_R(\bO) - F_L(\bO)) \left( \nabla \cL_{F_R} H(\bO) \right)^{\sf T}}{\cL_{F_L} H(\bO)}.
\label{eq:DFS}
\end{equation}

We now calculate the evolution time $T$ between $X^{(2)}$ and $X^{(1)}$.
Let $\varphi^R_t(X)$ denote the flow induced by $\dot{X} = F_R(X)$,
so $X^{(1)} = \varphi^R_T \left( X^{(2)} \right)$.
By \eqref{eq:gsFR}, the Taylor expansion of the flow about $(X,t) = (\bO,0)$ is
\begin{equation}
\varphi^R_t(X) = X + F_R(\bO) t + \tfrac{1}{2} \rD F_R(\bO) F_R(\bO) t^2
+ \rD F_R(\bO) X t + \cO \left( \left( \| X \| + |t| \right)^3 \right).
\label{eq:gsFlow}
\end{equation}
By substituting this into \eqref{eq:gsH} and \eqref{eq:gsLieR} we obtain
\begin{align}
H \left( \varphi^R_t(X) \right) &= \nabla H(\bO)^{\sf T} X + \tfrac{1}{2} X^{\sf T} \rD^2 H(\bO) X
+ \left( \nabla H(\bO)^{\sf T} \rD F_R(\bO) + F_R(\bO)^{\sf T} \rD^2 H(\bO) \right) X t \nonumber \\
&\quad+ \tfrac{1}{2} \left( \nabla H(\bO)^{\sf T} \rD F_R(\bO) F_R(\bO) + F_R(\bO)^{\sf T} \rD^2 H(\bO) F_R(\bO) \right) t^2
+ \cO \left( \left( \| X \| + |t| \right)^3 \right), \label{eq:gsHFlow} \\
\cL_{F_R} H \left( \varphi^R_t(X) \right) &= \left( \nabla \cL_{F_R} H(\bO) \right)^{\sf T} X
+ \cL_{F_R}^2 H(\bO) t + \cO \left( \left( \| X \| + |t| \right)^2 \right). \label{eq:gsLieRFlow}
\end{align}
Since $X^{(1)} \in \Omega$, we have
$\cL_{F_R} H \left( \varphi^R_T \left( X^{(1)} \right) \right) = 0$,
thus by \eqref{eq:gsLieRFlow},
\begin{equation}
T = -\frac{x_2}{\cL_{F_R}^2 H(\bO)} + \cO \left( \| X^{(2)} \|^2 \right),
\label{eq:gsT}
\end{equation}
where for brevity we write $x_2 = \left( \nabla \cL_{F_R} H(\bO) \right)^{\sf T} X^{(2)}$.
By substituting this into \eqref{eq:gsHFlow}, we obtain
\begin{align}
H \left( X^{(1)} \right) &= -\frac{x_2^2}{2 \cL_{F_R}^2 H(\bO)}
+ \cO \left( \| X^{(2)} \|^3 \right).
\label{eq:gsH1}
\end{align}
using also $H \left( X^{(2)} \right) = 0$ since $X^{(2)} \in \Sigma$.

We now compute $X^{(3)}$.
If this point is computed directly,
the required asymptotic calculations are relatively
lengthy as they include terms that are both linear and quadratic in $x_2$,
yet many of these terms do not
contribute to the linear term in \eqref{eq:gsFormula}.
So instead we first compute $X_{\rm vc} = \varphi^S_T \left( X^{(2)} \right)$,
where $\varphi^S_t(X)$ is the flow induced by $\dot{X} = F_S(X)$.
We then perform an adjustment to shift from $X_{\rm vc}$ to $X^{(3)}$.
Since we only compute the differences $X_{\rm vc} - X^{(1)}$ and $X^{(3)} - X_{\rm vc}$,
it is sufficient to retain only the first non-zero terms in each step of the asymptotics
and this greatly simplifies the calculations.
We call $X_{\rm vc}$ a {\em virtual counterpart},
as it may be virtual (i.e.~in the crossing region),
and mirrors the point $X^{(1)}$.

By \eqref{eq:gsFS}, the flow induced by $\dot{X} = F_S(X)$ has the Taylor expansion
\begin{equation}
\varphi^S_t(X) = X + F_R(\bO) t + \tfrac{1}{2} \rD F_S(\bO) F_R(\bO) t^2
+ \rD F_S(\bO) X t + \cO \left( \left( \| X \| + |t| \right)^3 \right),
\label{eq:gsFlowSliding}
\end{equation}
where $\rD F_S(\bO)$ is given by \eqref{eq:DFS}.
By substituting \eqref{eq:gsT} into \eqref{eq:gsFlow} and \eqref{eq:gsFlowSliding},
and subtracting the two expressions, we obtain
\begin{equation}
X_{\rm vc} = X^{(1)} - \frac{(F_R(\bO) - F_L(\bO)) x_2^2}
{2 \cL_{F_L} H(\bO) \cL_{F_R}^2 H(\bO)} + \cO \left( \| X^{(2)} \|^3 \right).
\label{eq:gsVC}
\end{equation}
Let $U$ denote the evolution time from $X_{\rm vc}$ to $X^{(3)}$,
so $X^{(3)} = \varphi^S_U \left( X_{\rm vc} \right)$.
Since $X^{(3)} \in \Omega$,
$\cL_{F_R} H \left( \varphi^S_U \left( X_{\rm vc} \right) \right) = 0$,
and thus by \eqref{eq:gsLieRFlow},
\begin{equation}
U = -\frac{x_{\rm vc}}{\cL_{F_R}^2 H(\bO)} + \cO \left( \| X_{\rm vc} \|^2 \right),
\nonumber
\end{equation}
where $x_{\rm vc} = \left( \nabla \cL_{F_R} H(\bO) \right)^{\sf T} X_{\rm vc}$.
Since $\cL_{F_R} H \left( X^{(1)} \right) = 0$, from \eqref{eq:gsVC} we obtain
\begin{equation}
U = -\frac{1}{2 \cL_{F_R}^2 H(\bO)} \left(
\frac{\cL_{F_L} \cL_{F_R} H(\bO)}{\cL_{F_L} H(\bO) \cL_{F_R}^2 H(\bO)} -
\frac{1}{\cL_{F_L} H(\bO)} \right) x_2^2
+ \cO \left( \| X^{(1)} \|^{\frac{3}{2}} \right),
\label{eq:gsU}
\end{equation}
where the error term is a consequence of the fact
that $\left\| X^{(2)} \right\| = \cO \left( \| X^{(1)} \|^{\frac{1}{2}} \right)$
for a given point $X^{(1)}$ near $\bO$.
Finally, $X^{(3)} = X_{\rm vc} + F_S(X_{\rm vc}) U + \cO \left( \left( \| X_{\rm vc} \| + |U| \right)^2 \right)$,
so by \eqref{eq:gsFS}, \eqref{eq:gsVC}, and \eqref{eq:gsU} we obtain
\begin{equation}
X^{(3)} = X^{(1)} - \frac{1}{2 \cL_{F_R}^2 H(\bO)}
\left( \frac{\cL_{F_L} \cL_{F_R} H(\bO)}{\cL_{F_L} H(\bO) \cL_{F_R}^2 H(\bO)} \,F_R(\bO)
- \frac{1}{\cL_{F_L} H(\bO)} \,F_L(\bO) \right) x_2^2 + \cO \left( \| X^{(1)} \|^{\frac{3}{2}} \right).
\nonumber
\end{equation}
By \eqref{eq:gsH1} this reduces to the desired formula \eqref{eq:gsFormula} with $X_{\rm fold} = \bO$.
\end{proof}

\section{The derivative of the discontinuity map}
\label{app:disc}

Here we derive the formula \eqref{eq:DQdisc}
for the derivative of $Q_{\rm disc}$
at the grazing point and in the limit $\nu \to 0$.
By \eqref{eq:tildeH2} and \eqref{eq:abc},
\begin{equation}
\tilde{H}(Y;0) = a Y_1 + b Y_2 + c Y_3 + \psi.
\label{eq:tildeH3}
\end{equation}
So by evaluating \eqref{eq:8p74} at the perturbed point \eqref{eq:perturbedPoint},
we obtain
\begin{equation}
Y' = \begin{bmatrix}
\frac{-a \psi}{a^2 + b^2} + \ee_1 \\
\frac{-b \psi}{a^2 + b^2} + \frac{b}{a} \,\ee_1 - \frac{c \gamma_0}{a \beta_0} \,\ee_3 \\
\ee_3
\end{bmatrix} + Z(Y;0) \left( \tfrac{a^2 + b^2}{a} \,\ee_1
+ \big( 1 - \tfrac{b \gamma_0}{a \beta_0} \big) \ee_3 \right)
+ \cO \left( \left( |\ee_1| + |\ee_3| \right)^{\frac{3}{2}} \right).
\label{eq:DQdiscProof10}
\end{equation}
The coefficients of the $\ee_1$ and $\ee_3$ terms in \eqref{eq:DQdiscProof10} are
\begin{equation}
\begin{bmatrix} \frac{\partial Y_1'}{\ee_1} & \frac{\partial Y_1'}{\ee_3} \\[1.2mm]
\frac{\partial Y_3'}{\ee_1} & \frac{\partial Y_3'}{\ee_3} \end{bmatrix}
= \begin{bmatrix}
1 + \frac{a^2 + b^2}{a} \,Z_1(G(0);0) &
\left( 1 - \frac{b \gamma_0}{a \beta_0} \right) Z_1(G(0);0) \\
\frac{a^2 + b^2}{a} \,Z_3(G(0);0) &
1 + \left( 1 - \frac{b \gamma_0}{a \beta_0} \right) Z_3(G(0);0) \end{bmatrix},
\label{eq:DQdiscProof20}
\end{equation}
and notice we have been able to set $\ee_1 = \ee_3 = 0$ inside $Z$.
To evaluate $Z(G(0);0)$, we introduce the scaled vector field $J_L(Y;\nu) = \nu \tilde{F}_L(Y;\nu)$,
because $\tilde{F}_L(Y;\nu)$ blows up in the limit $\nu \to 0$.
In terms of $J_L$ and $\tilde{F}_R$, \eqref{eq:8p74vector} takes the form
\begin{equation}
Z = \frac{\cL_{J_L} \cL_{\tilde{F}_R} \tilde{H}}
{\big( \cL_{\tilde{F}_R}^2 \tilde{H} \big) \big( \cL_{J_L} \tilde{H} \big)} \,\tilde{F}_R
- \frac{1}{\cL_{J_L} \tilde{H}} \,J_L \,.
\label{eq:8p74vectorB}
\end{equation}
By \eqref{eq:realJordanForm}, \eqref{eq:tildeFL2}, \eqref{eq:tildeFR2}, and \eqref{eq:pqr},
\begin{align}
J_L(Y;0) &= \begin{bmatrix} p \\ q \\ r \end{bmatrix}, &
\tilde{F}_R(Y;0) &= \begin{bmatrix} \beta_0 Y_2 \\ -\beta_0 Y_1 \\ \gamma_0 Y_3 \end{bmatrix},
\nonumber
\end{align}
thus
\begin{align}
\cL_{J_L} \tilde{H} &= a p + b q + c r, \nonumber \\
\cL_{\tilde{F}_R}^2 \tilde{H} &= -a \beta_0^2 Y_1 - b \beta_0^2 Y_2 + c \gamma_0^2 Y_3 \,, \nonumber \\
\cL_{J_L} \cL_{\tilde{F}_R} \tilde{H} &= -b p \beta_0 + a \beta_0 q + c r \gamma_0 \,, \nonumber
\end{align}
using also \eqref{eq:tildeH3}.
By substituting these into \eqref{eq:8p74vectorB}, we obtain
\begin{equation}
Z(G(0);0) = \frac{-1}{a^2 + b^2} \begin{bmatrix} a \\ b \\ 0 \end{bmatrix}
+ \frac{r}{\beta_0 \left( a^2 + b^2 \right)(a p + b q + c r)}
\begin{bmatrix}
(a \beta_0 - b \gamma_0) c \\
(b \beta_0 + a \gamma_0) c \\
-\left( a^2 + b^2 \right) \beta_0
\end{bmatrix},
\label{eq:8p74vectorC}
\end{equation}
using also the formula \eqref{eq:grazPoint2} for $G$.
By substituting the first and third components
of \eqref{eq:8p74vectorC} into \eqref{eq:DQdiscProof20},
we arrive at the desired formula
\begin{equation}
\begin{bmatrix} \frac{\partial Y_1'}{\ee_1} & \frac{\partial Y_1'}{\ee_3} \\[1.2mm]
\frac{\partial Y_3'}{\ee_1} & \frac{\partial Y_3'}{\ee_3} \end{bmatrix}
= \begin{bmatrix}
\frac{(a \beta_0 - b \gamma_0) c r}{a \beta_0 (a p + b q + c r)} &
\frac{-(a \beta_0 - b \gamma_0) c}{\beta_0 \left( a^2 + b^2 \right)(a p + b q + c r)}
\left( a p + b q + \frac{b c r \gamma_0}{a \beta_0} \right) \\[1.2mm]
\frac{-\left( a^2 + b^2 \right) r}{a (a p + b q + c r)} &
1 - \frac{(a \beta_0 - b \gamma_0) c r}{a \beta_0 (a p + b q + c r)}
\end{bmatrix}.
\nonumber
\end{equation}

{\footnotesize
\bibliographystyle{unsrt}
\bibliography{BEBHB3d_arXivBIB}

@book{DiBu08,
	author = {di Bernardo, M. and Budd, C.J. and Champneys, A.R.
		and Kowalczyk, P.},
	title = {Piecewise-smooth Dynamical Systems. Theory and Applications.},
	publisher = {Springer-Verlag},
	address = {New York},
	year = 2008,
}

@book{Fi88,
	author = {Filippov, A.F.},
	title = {Differential Equations with Discontinuous Righthand Sides.},
	publisher = {Kluwer Academic Publishers.},
	address = {Norwell},
	year = 1988,
}

@book{Je18b,
	author = {Jeffrey, M.R.},
	title = {Hidden Dynamics. The Mathematics of Switches, Decisions and Other
		Discontinuous Behaviour.},
	publisher = {Springer},
	address = {New York},
	year = 2018,
}

@article{JoRa99,
	author = {Johansson, K.H. and Rantzer, A. and {\AA}str\"{o}m, K.J.},
	title = {Fast switches in relay feedback systems.},
	journal = {Automatica},
	volume = 35,
	pages = {539-552},
	year = 1999,
}

@book{BlCz99,
	author = {Blazejczyk-Okolewska, B. and Czolczynski, K. and
		Kapitaniak, T. and Wojewoda, J.},
	title = {Chaotic Mechanics in Systems with Impacts and Friction},
	publisher = {World Scientific},
	address = {Singapore},
	year = 1999,
}

@article{FeGu98,
	author = {Feeny, B. and Guran, A. and Hinrichs, N. and Popp, K.},
	title = {A Historical Review of Dry Friction and Stick-Slip Phenomena.},
	journal = {Appl. Mech. Rev.},
	volume = 51,
	number = 5,
	pages = {321-341},
	year = 1998,
}

@article{DeGr07,
	author = {Dercole, F. and Gragnani, A. and Rinaldi, S.},
	title = {Bifurcation analysis of piecewise smooth ecological models.},
	journal = {Theor. Popul. Biol.},
	volume = 72,
	pages = {197-213},
	year = 2007,
}

@book{PuSu06,
	editor = {Puu, T. and Sushko, I.},
	title = {Business Cycle Dynamics: Models and Tools.},
	publisher = {Springer-Verlag},
	address = {New York},
	year = 2006,
}

@article{DeDe11,
	author = {Dercole, F. and Della Rossa, F. and Colombo, A. and
		Kuznetsov, Yu.A.},
	title = {Two Degenerate Boundary Equilibrium Bifurcations in Planar
		{F}ilippov Systems.},
	journal = {SIAM J. Appl. Dyn. Syst.},
	volume = 10,
	number = 4,
	pages = {1525-1553},
	year = 2011,
}

@article{DiPa08,
	author = {di Bernardo, M. and Pagano, D.J. and Ponce, E.},
	title = {Nonhyperbolic boundary equilibrium bifurcations in
		planar {F}ilippov systems: {A} case study approach.},
	journal = {Int. J. Bifurcation Chaos},
	volume = 18,
	number = 5,
	pages = {1377-1392},
	year = 2008,
}

@article{GuSe11,
	author = {Guardia, M. and Seara, T.M. and Teixeira, M.A.},
	title = {Generic bifurcations of low codimension of planar
		{F}ilippov Systems.},
	journal = {J. Diff. Eq},
	volume = 250,
	pages = {1967-2023},
	year = 2011,
}

@book{Si10,
	author = {Simpson, D.J.W.},
	title = {Bifurcations in Piecewise-Smooth Continuous Systems.},
	series = {Nonlinear Science},
	volume = 70,
  publisher = {World Scientific},
  address = {Singapore},
	year = 2010,
}

@article{SiKo09,
	author = {Simpson, D.J.W. and Kompala, D.S. and Meiss, J.D.},
	title = {Discontinuity Induced Bifurcations in a Model
		of ${S}$\hspace{-.3mm}{\em accharomyces cerevisiae}.},
	journal = {Math. Biosci.},
	volume = 218,
	number = 1,
	pages = {40-49},
	year = 2009,
}

@article{SiMe08,
	author = {Simpson, D.J.W. and Meiss, J.D.},
	title = {Unfolding a Codimension-Two Discontinuous
		{A}ndronov-{H}opf Bifurcation.},
	journal = {Chaos},
	volume = 18,
	number = 3,
	pages = {033125},
	year = 2008,
}

@article{DiBu01,
	author = {di Bernardo, M. and Budd, C.J. and Champneys, A.R.},
	title = {Normal form maps for grazing bifurcations in $n$-dimensional
		piecewise-smooth dynamical systems.},
	journal = {Phys. D},
	volume = 160,
	pages = {222-254},
	year = 2001,
}

@article{JeCh10,
	author = {Jeffrey, M.R. and Champneys, A.R. and di Bernardo, M. and
		Shaw, S.W.},
	title = {Catastrophic sliding bifurcations and onset of oscillations in a
		superconducting resonator.},
	journal = {Phys. Rev. E},
	volume = 81,
	pages = {016213},
	year = 2010,
}

@article{JeHo11,
	author = {Jeffrey, M.R. and Hogan, S.J.},
	title = {The Geometry of Generic Sliding Bifurcations.},
	journal = {SIAM Rev.},
	volume = 53,
	number = 3,
	pages = {505-525},
	year = 2011,
}

@article{NuYo92,
	author = {Nusse, H.E. and Yorke, J.A.},
	title = {Border-Collision Bifurcations Including
		``Period Two to Period Three'' for Piecewise Smooth Systems.},
	journal = {Phys. D},
	volume = 57,
	pages = {39-57},
	year = 1992,
}

@article{AvSc12,
	author = {Avrutin, V. and Schanz, M. and Banerjee, S.},
	title = {Occurrence of multiple attractor bifurcations in the
		two-dimensional piecewise linear normal form map.},
	journal = {Nonlin. Dyn.},
	volume = 67,
	pages = {293-307},
	year = 2012,
}

@article{BaGr99,
	author = {Banerjee, S. and Grebogi, C.},
	title = {Border collision bifurcations in two-dimensional piecewise
		smooth maps.},
	journal = {Phys. Rev. E},
	volume = 59,
	number = 4,
	pages = {4052-4061},
	year = 1999,
}

@article{FaSi23,
	author = {Fatoyinbo, H.O. and Simpson, D.J.W.},
	title = {A synopsis of the non-invertible, two-dimensional, border-collision
		normal form with applications to power converters.},
	journal = {Int. J. Bifurcation Chaos},
	volume = 33,
	number = 8,
	pages = {2330019},
	year = 2023,
}

@article{GhMc24,
	author = {Ghosh, I. and McLachlan, R.I. and Simpson, D.J.W.},
	title = {The Bifurcation Structure within Robust Chaos for Two-Dimensional
		Piecewise-Linear Maps.},
	journal = {Commun. Nonlin. Sci. Numer. Simul.},
	volume = 134,
	pages = {108025},
	year = 2024,
}

@article{Gl16e,
	author = {Glendinning, P.},
	title = {Bifurcation from stable fixed point to {2D} attractor
		in the border collision normal form.},
	journal = {IMA J. Appl. Math.},
	volume = 81,
	number = 4,
	pages = {699-710},
	year = 2016,
}

@misc{SuGa06,
	author = {Sushko, I. and Gardini, L.},
	title = {Center Bifurcation for a Two-Dimensional Piecewise
		Linear Map.},
	note = {In \cite{PuSu06}, pages 49-78},
}

@article{ZhMo08,
	author = {Zhusubaliyev, Z.T. and Mosekilde, E. and De, S. and
		Banerjee, S.},
	title = {Transitions from phase-locked dynamics to chaos in a
		piecewise-linear map.},
	journal = {Phys. Rev. E},
	volume = 77,
	pages = {026206},
	year = 2008,
}

@article{HaPo91,
	author = {Hastings, A. and Powell, T.},
	title = {Chaos in a three-species food chain.},
	journal = {Ecology},
	volume = 72,
	number = 3,
	pages = {896-903},
	year = 1991,
}

@book{Hi76,
	author = {Hirsch, M.W.},
	title = {Differential Topology.},
	publisher = {Springer-Verlag},
	address = {New York},
	year = 1976,
}

@article{DiNo08,
	author = {di Bernardo, M. and Nordmark, A. and Olivar, G.},
	title = {Discontinuity-induced bifurcations of equilibria in
		piecewise-smooth and impacting dynamical systems.},
	journal = {Phys. D},
	volume = 237,
	pages = {119-136},
	year = 2008,
}

@book{Gl99,
	author = {Glendinning, P.},
	title = {Stability, Instability and Chaos:
		An Introduction to the Theory of Nonlinear Differential
		Equations.},
	publisher = {Cambridge University Press},
	address = {New York},
	year = 1999,
}

@book{Ku04,
	author = {Kuznetsov, Yu.A.},
	title = {Elements of Bifurcation Theory.},
	publisher = {Springer-Verlag},
	address = {New York},
	volume = 112,
	series = {Appl. Math. Sci.},
	edition = {3rd},
	year = 2004,
}

@book{Me07,
	author = {Meiss, J.D.},
	title = {Differential Dynamical Systems.},
	publisher = {SIAM},
	address = {Philadelphia},
	year = 2007,
}

@article{DiKo02,
	author = {di Bernardo, M. and Kowalczyk, P. and Nordmark, A.},
	title = {Bifurcations of dynamical systems with sliding:
		Derivation of normal-form mappings.},
	journal = {Phys. D},
	volume = 170,
	pages = {175-205},
	year = 2002,
}

@article{Si16,
	author = {Simpson, D.J.W.},
	title = {Border-collision bifurcations in $\mathbb{R}^N$.},
	journal = {SIAM Rev.},
	volume = 58,
	number = 2,
	pages = {177-226},
	year = 2016,
}

@article{Si18d,
	author = {Simpson, D.J.W.},
	title = {A general framework for boundary equilibrium bifurcations of {F}ilippov systems.},
	journal = {Chaos},
	volume = 28,
	number = 10,
	pages = {103114},
	year = 2018,
}

@article{GlSi21,
	author = {Glendinning, P.A. and Simpson, D.J.W.},
	title = {A constructive approach to robust chaos using invariant manifolds and expanding cones.},
	journal = {Discrete Contin. Dyn. Syst.},
	volume = 41,
	number = 7,
	pages = {3367-3387},
	year = 2021,
}

@article{Ko05,
	author = {Kowalczyk, P.},
	title = {Robust chaos and border-collision bifurcations 
		in non-invertible piecewise-linear maps.},
	journal = {Nonlinearity},
	volume = 18,
	pages = {485-504},
	year = 2005,
}

@inproceedings{Mi80,
	author = {Misiurewicz, M.},
	title = {Strange attractors for the {L}ozi mappings.},
	editor = {Helleman, R.G.},
	booktitle = {Nonlinear dynamics, Annals of the New York
		Academy of Sciences},
	publisher = {Wiley},
	address = {New York},
	pages = {348-358},
	year = 1980,
}

@article{GlSi20b,
	author = {Glendinning, P.A. and Simpson, D.J.W.},
	title = {Robust chaos and the continuity of attractors.},
	journal = {Trans. Math. Appl.},
	volume = 4,
	number = 1,
	pages = {tnaa002},
	year = 2020,
}

@article{GhSi22,
	author = {Ghosh, I. and Simpson, D.J.W.},
	title = {Renormalisation of the Two-Dimensional Border-Collision
			Normal Form.},
	journal = {Int. J. Bifurcation Chaos},
	volume = 32,
	number = 12,
	pages = {2250181},
	year = 2022,
}

@article{Va26,
	author = {van der Pol, B.},
	title = {On relaxation-oscillations.},
	journal = {Phil. Mag.},
	volume = 2,
	pages = {978-992},
	year = 1926,
}

@article{Si22b,
	author = {Simpson, D.J.W.},
	title = {Twenty {H}opf-like bifurcations in piecewise-smooth dynamical systems.},
	journal = {Phys. Rep.},
	volume = 970,
	pages = {1-80},
	year = 2022,
}

@article{AbRo94,
	author = {Abrams, P.A. and Roth, J.D.},
	title = {The effects of enrichment of three-species food chains with nonlinear
		functional responses.},
	journal = {Ecology},
	volume = 75,
	number = 4,
	pages = {1118-1130},
	year = 1994,
}

@article{KuRi96,
	author = {Kuznetsov, Yu.A. and Rinaldi, S.},
	title = {Remarks on Food Chain Dynamics.},
	journal = {Math. Biosci.},
	volume = 134,
	pages = {1-33},
	year = 1996,
}

@article{McYo95,
	author = {McCann, K. and Yodzis, P.},
	title = {Bifurcation Structure of a Three-Species Food Chain Model.},
	journal = {Theor. Popul. Biol.},
	volume = 48,
	pages = {93-125},
	year = 1995,
}

@article{Ko98,
	author = {Kogan, M.},
	title = {Integrated Pest Management: {H}istorical Perspectives and
		Contemporary Developments.},
	journal = {Annu. Rev. Entomol.},
	volume = 43,
	pages = {243-270},
	year = 1998,
}

@article{Jo04,
	author = {Jones, R.A.C.},
	title = {Using epidemiological information to develop effective integrated
		virus disease management strategies.},
	journal = {Virus Res.},
	volume = 100,
	pages = {5-30},
	year = 2004,
}

@article{ZhTa22,
	author = {Zhou, H. and Tang, S.},
	title = {Bifurcation dynamics on the sliding vector field of a {F}ilippov
		ecological system.},
	journal = {Appl. Math. Comput.},
	volume = 424,
	pages = {127052},
	year = 2022,
}

@article{HaAr21,
	author = {Hamdallah, S.A.A. and Arafa, A.A. and Tang, S. and Xu, Y.},
	title = {Complex Dynamics of a {F}ilippov Three-Species Food Chain Model.},
	journal = {Int. J. Bifurcation Chaos},
	volume = 31,
	number = 5,
	pages = {2150074},
	year = 2021,
}

@article{NoKo06,
	author = {Nordmark, A.B. and Kowalczyk, P.},
	title = {A codimension-two scenario of sliding solution in
		grazing-sliding bifurcations.},
	journal = {Nonlinearity},
	volume = 19,
	number = 1,
	pages = {1-26},
	year = 2006,
}

@book{GlJe19,
	author = {Glendinning, P. and Jeffrey, M.R.},
	title = {An Introduction to Piecewise Smooth Dynamics.},
	publisher = {Birkhauser},
	address = {Boston},
	year = 2019,
}

@book{Ru17,
	author = {Ruette, S.},
	title = {Chaos on the interval.},
	publisher = {American Mathematical Society},
	address = {Providence, RI},
	year = 2017,
}
}

\end{document}